\documentclass[journal,twoside]{IEEEtran}

\usepackage{enumitem}
\usepackage[utf8]{inputenc}
\usepackage{graphicx}
\usepackage{textcomp}
\usepackage{xcolor}
\usepackage{float}
\usepackage{comment}
\usepackage{algorithmic}
\usepackage{algorithm}
\usepackage{amsmath,amssymb}
\usepackage{cite}
\usepackage{tikz}
\usetikzlibrary{arrows.meta,positioning,calc,decorations.pathreplacing,patterns,shapes.geometric,shapes.misc,backgrounds}

\usepackage{style} 
\usepackage{mathtools}
\usepackage{xpatch}

\usepackage[numbered]{bookmark}
\usepackage{hyperref} 

\providecommand{\Du}{\mathfrak{D}}                     
\providecommand{\Lf}[2]{L_{#1} #2}                     
\DeclareMathOperator*{\lexmin}{lex\text{-}min}
\newcommand{\minRe}{\min\nolimits_{\Re}}               
\newcommand{\maxRe}{\max\nolimits_{\Re}}               

\allowdisplaybreaks
\def\BibTeX{{\rm B\kern-.05em{\sc i\kern-.025em b}\kern-.08em
		T\kern-.1667em\lower.7ex\hbox{E}\kern-.125emX}}

\begin{document}
	\IEEEoverridecommandlockouts
	\IEEEpubid{\makebox[\columnwidth]{This work has been submitted to the IEEE for possible publication.\hfill}\hspace{\columnsep}\makebox[\columnwidth]{ }}
	
	\title{Deterministic Non-Smooth Safety \\via Dual-Algebraic Control Barrier Functions}
	
	\author{Mohammadreza Kamaldar
		\thanks{Manuscript received June 5, 2026. (Corresponding author: Mohammadreza Kamaldar.)}
		\thanks{Mohammadreza Kamaldar is with the Department of Mechanical, Aerospace, and Biomedical Engineering, University of South Alabama, Mobile, AL 36688 USA (e-mail: mkamaldar@southalabama.edu).}}
	
	\maketitle

	\begin{abstract}
		This paper presents a dual-algebraic framework for control barrier functions (CBFs) that guarantees deterministic execution using exclusively elementary arithmetic. We develop this deterministic approach to solve a fundamental bottleneck in safety-critical control: pointwise minima compose intersecting safe sets, but generate non-smooth boundaries where standard Lie derivatives fail. Existing mathematical workarounds inject approximation bias, probabilistic non-determinism, or combinatorial execution delays that impede hard real-time hardware certification. By embedding the system state and vector field into the dual-number ring, our method extracts both the composite barrier value and its exact directional derivative in a single evaluation. The standard floating-point minimum deterministically isolates a single vertex of the Clarke generalized gradient for the quadratic-program solver. We prove this selected vertex constitutes a Clarke subgradient and the resulting simultaneous-enforcement safety filter guarantees forward invariance. The arithmetic overhead remains a fixed constant factor, independent of state dimension and constraint count. We extend this framework to finite $\min$/$\max$ Boolean compositions, for which enforcement of the routed constraint of each $\delta$-active clause guarantees forward invariance, and to systems of higher relative degree, for which a bivariate truncated-dual evaluation extracts the control coupling without symbolic differentiation. Three numerical examples illustrate the computational scaling.
	\end{abstract}
	\IEEEpubidadjcol
	\begin{IEEEkeywords}
		Control barrier functions, non-smooth analysis, dual numbers, Clarke subdifferential, safety-critical control, forward invariance.
	\end{IEEEkeywords}
	
	\section{Introduction}
	\label{sec:intro}
	
	Safety-critical autonomous systems must guarantee that the state remains within a prescribed safe region at all times.
	The reachability and set-invariance literature provides foundational tools for characterizing safe sets and synthesizing controllers that enforce state constraints, including Hamilton-Jacobi reachability for hybrid systems \cite{lygeros1999reachability}, viability theory \cite{aubin2009viability}, and set-invariance methods \cite{blanchini1999set}.
	Control barrier functions (CBFs), introduced in \cite{wieland2007constructive} and formalized as quadratic-program-based safety controllers in \cite{ames2017cbf,ames2019cbf_review}, with roots in barrier-certificate \cite{prajna2007framework} and barrier-Lyapunov \cite{tee2009barrier} formulations, embed forward invariance of a safe set into a real-time quadratic program (QP) \cite{boyd2004convex} by means of a pointwise affine inequality.
	CBF-based quadratic programs have been applied to adaptive cruise control \cite{ames2014cbf}, bipedal locomotion \cite{hsu2015control}, multi-robot coordination \cite{wang2017safety}, and robot manipulation \cite{rauscher2016constrained}, and have been extended to high relative degree \cite{nguyen2016exponential,xiao2019high,xiao2022high}, input-to-state safety \cite{kolathaya2018input}, robustness \cite{xu2015robustness}, temporal-logic tasks \cite{lindemann2019control}, and sampled-data implementation \cite{breeden2021control}.
	When the barrier function $h$ is continuously differentiable, this inequality involves the Lie derivative $\Lf{f}{h}(x) \triangleq \nabla h(x) \cdot f(x)$ along the system vector field, which is well defined at every point.
	
	In practice, however, the safe region is most often described as the intersection of $p$ smooth constraint sets, and the natural composite barrier is the pointwise minimum
	\begin{equation}
		h(x) \triangleq \min_{i \in \{1,\ldots,p\}} h_i(x),\nn
		\label{eq:intro_min}
	\end{equation}
	where each $h_i$ is continuously differentiable. Although $h$ is locally Lipschitz, it fails to be differentiable at every state where two or more constraints attain the minimum, and the classical gradient $\nabla h(x)$ is undefined at such points. The standard CBF formulation cannot therefore be applied directly along trajectories that cross the non-smooth locus.
	
	Existing approaches circumvent this obstruction through smooth approximations, stochastic formulations, or exact geometric formulations, and each carries a practical penalty with respect to the design objectives of this paper. Smooth surrogates replace the exact minimum with a differentiable under-approximation, such as the log-sum-exp function $h_\tau(x) \triangleq -\tau \ln \sum_{i=1}^p e^{-h_i(x)/\tau}$, which converges to $h(x)$ as $\tau \to 0^+$ and is used with $\tau = 1$ in \cite{lindemann2019control}. Tuning the parameter $\tau$ involves a tradeoff: small values produce numerical ill-conditioning, while large values shift the represented boundary inward and enlarge the approximation error. Randomized smoothing replaces the non-smooth function by its expectation under a random perturbation and estimates the gradient of that expectation by sampling \cite{duchi2012randomized}, so the resulting gradient is a stochastic estimate. Stochastic CBF formulations model the dynamics as stochastic differential equations and certify invariance with probability one \cite{vahs2024nonsmooth}; these formulations address model uncertainty, and their certificates are stated for a stochastic execution model rather than the deterministic one required here. Exact formulations confront the geometry directly and carry computational or feasibility considerations of their own. Enforcing one constraint per component function at every control cycle \cite{glotfelter2017nonsmooth} removes active-set identification, and the constraint count of the quadratic program then scales with $p$ together with the feasibility requirement. Mixed-integer quadratic programs (MIQPs) select, through binary variables, which constraints of a disjunction to enforce \cite{chen2018obstacle}, and their worst-case solve time grows exponentially with the number of binary variables. Constructing the Clarke generalized gradient $\partial h(x)$ \cite{clarke1990optimization} provides the exact directional derivatives and requires the convex hull of the active gradients; the hull of $p$ points in $\BBR^n$ has up to $O(p^{\lfloor n/2 \rfloor})$ facets, and hull algorithms carry a corresponding worst-case cost \cite{barber1996quickhull}, so the execution time fluctuates with the local boundary geometry, which conflicts with the bounded-time requirements of safety standards such as DO-178C.

	Dual-algebraic evaluation bypasses these limitations. Evaluating the composite barrier within the dual-number ring $\BBD \triangleq \BBR[\epsilon]/(\epsilon^2)$ extracts the value $h(x)$ and the exact Lie derivative $\Lf{f}{h}(x)$ simultaneously in a single forward pass. While dual-number arithmetic is standard in algorithmic differentiation \cite{griewank2008evaluating}, its application as a deterministic engine for non-smooth safety-critical control is, to the best of our knowledge, new. When the dual-extended evaluation reaches a pointwise $\min$ of dual scalars, the standard floating-point comparison examines only the real parts, natively implementing a lexicographic order on $\BBD$. At every state, including states on the non-smooth locus, this comparison deterministically routes a single active constraint and returns its exact Lie derivative. We prove this selected gradient is a vertex of the Clarke generalized gradient at $x$, and that simultaneous enforcement of the $\delta$-active constraints yields forward invariance of the safe set.

	This paper develops the dual-algebraic evaluation of a composite barrier $h \triangleq \min_{i\in\{1,\ldots,p\}} h_i$ and proves that a single evaluation returns both the value $h(x)$ and a vertex of the Clarke subdifferential at every $x \in \SD$, including non-smooth points (Theorem~\ref{thm:vertex_selection}). Enforcing this single vertex does not by itself guarantee safety; we prove instead that a quadratic program enforcing all $\delta$-active constraints simultaneously, with the dual-derived Lie derivatives, renders the safe set forward invariant (Theorem~\ref{thm:forward_invariance}). The framework extends to barriers formed by finite pointwise $\min$ and $\max$ compositions, for which the evaluation returns a generator of the active-gradient polytope (Theorem~\ref{thm:composition}) and for which, in conjunctive form, enforcement of the routed constraint of each $\delta$-active clause renders the safe set forward invariant (Theorem~\ref{thm:cnf_invariance}). Using truncated polynomial duals $\BBR[\epsilon]/(\epsilon^{r+1})$, it extends to systems of arbitrary relative degree, for which a single seeded dual evaluation produces the iterated Lie derivatives $L_f^j h$, $j\in\{1,\ldots,r\}$ (Theorem~\ref{thm:higher_rel_deg}), and a bivariate truncated-dual evaluation produces the control coupling $\Lf{G}{L_f^{r-1}h}$ (Proposition~\ref{prop:coupling}). Finally, the dual evaluation increases the operation count of evaluating $h$ by a fixed constant factor, independent of the state dimension $n$, the constraint count $p$, and the active-set cardinality $|\SA(x)|$ (Proposition~\ref{prop:complexity}), which yields the deterministic worst-case execution time required for hard real-time certification.

	The remainder of the paper is organized as follows. Section~\ref{sec:notation} fixes the notation and reviews CBFs and the Clarke subdifferential. Section~\ref{sec:problem} states the problem and the design objectives precisely. Section~\ref{sec:dual_algebra} develops the dual-number algebra. Section~\ref{sec:main_results} contains the main results on vertex selection and forward invariance. Section~\ref{sec:extensions} extends the framework to compositions and to higher relative degree. Section~\ref{sec:QP} presents the QP-based safety filter and the complexity analysis. Section~\ref{sec:examples} illustrates the framework with three examples. Section~\ref{sec:discussion} compares the approach with existing methods and discusses its embedded realization, and Section~\ref{sec:conclusion} concludes and identifies future directions.

	\section{Notation and Preliminaries}
	\label{sec:notation}
	
	Let $\BBN \triangleq \{0,1,2,\ldots\}$ and $\BBZ_+ \triangleq \BBN \setminus \{0\}$. Let $\BBR$ denote the set of real numbers, and let $\|\cdot\|$ be a norm on $\BBR^n$. For a vector $x \in \BBR^n$, $x_{(i)}$ isolates its $i$-th component. For a matrix $A\in \BBR^{n \times m}$, $\col_j(A) \in \BBR^n$ extracts its $j$-th column. For all $i\in\{1,\ldots,n\}$, let $e_i\triangleq \col_{i}(I_n)$. 
	
	For a finite set $\SSS \subset \BBR^n$, $\Co \SSS$ denotes its closed convex hull, $\bd \SSS$ its topological boundary, and $|\SSS|$ its cardinality.
	
	A continuous function $\alpha \colon (-b,a) \to \BBR$, where $a,b > 0$, belongs to the extended class-$\SK$, denoted $\SK_e$, if it is  increasing and satisfies $\alpha(0) = 0$. 
	
	The scalar $j$-th derivative of a single-variable function $\phi \colon \BBR \to \BBR$ evaluated at $a \in \BBR$ is denoted $\phi^{(j)}(a)$. For a continuously differentiable multivariable function $h \colon \BBR^n \to \BBR$, the first derivative at $x \in \BBR^n$ is the gradient row vector
	\begin{equation}
		\nabla h(x) \triangleq \begin{bmatrix}
			\tfrac{\partial h}{\partial x_{(1)}}(x) & \ldots & \tfrac{\partial h}{\partial x_{(n)}}(x)
		\end{bmatrix} \in \BBR^{1 \times n}. \nn
	\end{equation}
	For a $k$-times continuously differentiable function $h \colon \BBR^n \to \BBR$, the symbol $D^k h(x)$ denotes the symmetric $k$-th order derivative tensor at $x$, with components
	\begin{equation}
		(D^k h(x))_{i_1, \ldots, i_k} \triangleq \frac{\partial^k h}{\partial x_{(i_1)} \cdots \partial x_{(i_k)}}(x),  \nn
	\end{equation}
	where $i_1,\ldots,i_k \in \{1,\ldots,n\}$.
	The bracket notation $D^k h(x)[v]^k$ denotes the contraction of this tensor against $k$ copies of $v \in \BBR^n$,
	\begin{equation}
		D^k h(x)[v]^k \triangleq \sum_{i_1=1}^n \cdots \sum_{i_k=1}^n \frac{\partial^k h}{\partial x_{(i_1)} \cdots \partial x_{(i_k)}}(x)\, v_{(i_1)} \cdots v_{(i_k)}. \nn
	\end{equation}
	For $k=1$, this reduces to the directional derivative $D^1 h(x)[v]^1 = \nabla h(x)\, v$.
	
	If $h \colon \BBR^n \to \BBR$ is continuously differentiable and $f \colon \BBR^n \to \BBR^n$, $G \colon \BBR^n \to \BBR^{n\times m}$ are continuous, then the Lie derivatives of $h$ along $f$ and $G$ are the respective spatial projections
	\begin{align}
		\Lf{f}{h}(x) &\triangleq \nabla h(x)\, f(x) \in \BBR, \nn\\ 
		\Lf{G}{h}(x) &\triangleq \nabla h(x)\, G(x) \in \BBR^{1\times m}. \nn
	\end{align}
	The iterated Lie derivative follows $L_f^0 h \triangleq h$ and $L_f^j h \triangleq L_f(L_f^{j-1} h)$, where $j \in \BBZ_+$.
	
	The dual-number ring $\BBD \triangleq \BBR[\epsilon]/(\epsilon^2)$ is formalized in Section~\ref{sec:dual_algebra}. Throughout this paper, $\Re(z)$ and $\Du(z)$ extract the real and dual parts of $z \in \BBD$, respectively, and $\tilde \phi \colon \BBD \to \BBD$ evaluates the exact dual extension of a smooth function $\phi \colon \BBR \to \BBR$.
	
	We consider the control-affine system
	\begin{equation}
		\dot{x} = f(x) + G(x) u, \label{eq:system}
	\end{equation}
	where $x \in \SD$ is the state, $\SD \subseteq \BBR^n$ is open, $u \in \BBR^m$ is the control input, $f \colon \SD \to \BBR^n$ is locally Lipschitz, and $G \colon \SD \to \BBR^{n \times m}$ is continuous.

	\subsection{Control Barrier Functions}
	
	\begin{definition}
		\label{def:safe_set}
		Let $h \colon \SD \to \BBR$ be continuous. The \textit{safe set} associated with $h$ is
		\begin{equation}
			\SC \triangleq \{x \in \SD \colon h(x) \geq 0\}, \nn
		\end{equation}
		and the \textit{boundary} of $\SC$ is $\bd \SC \triangleq \{x \in \SD \colon h(x) = 0\}$.
	\end{definition}
	
	\begin{definition}
		\label{def:CBF}
		Let $h \colon \SD \to \BBR$ be continuously differentiable. Then, $h$ is a \textit{control barrier function} (CBF) for \eqref{eq:system} on $\SD$ if there exists $\alpha \in \SK_e$ such that, for all $x \in \SD$,
		\begin{equation}
			\sup_{u \in \BBR^m} \Big(\, \Lf{f}{h}(x) + \Lf{G}{h}(x) u + \alpha(h(x)) \,\Big) \geq 0.\nn
		\end{equation}
	\end{definition}
	The following result restates \cite[Cor.\ 2]{ames2017cbf} with input set $\BBR^m$.
	\begin{fact}
		\label{fact:CBF_invariance}
		Let $h \colon \SD \to \BBR$ be continuously differentiable and a CBF for \eqref{eq:system} on $\SD$. If $u \colon \SD \to \BBR^m$ is a locally Lipschitz controller satisfying, for all $x \in \SD$,
		\begin{equation}
			\Lf{f}{h}(x) + \Lf{G}{h}(x) u(x) + \alpha(h(x)) \geq 0, \nn
		\end{equation}
		then $\SC$ is forward invariant with respect to the closed-loop system $\dot{x} = f(x) + G(x) u(x)$.
	\end{fact}

	\subsection{Clarke Subdifferential}
	
	Along trajectories that reach the non-smooth locus of $h$, the closed-loop vector field is discontinuous, and solutions are understood in the sense of Filippov \cite{filippov1988differential,cortes2008discontinuous}. The generalized gradient below provides the corresponding first-order object \cite{clarke1990optimization}.
	
	\begin{definition}
		\label{def:Clarke}
		Let $\SD \subset \BBR^n$ be an open set, and let $h \colon \SD \to \BBR$ be locally Lipschitz continuous. 
		Let $\Omega_h \subset \SD$ denote the set of Lebesgue measure zero where the gradient $\nabla h$ is undefined. For all $x \in \SD$, the \textit{Clarke generalized gradient} of $h$ at $x$ is
		\begin{align}
			\partial h(x) \triangleq \Co \Big\{ v \in \BBR^n \colon& \text{there exists~} \{x_i\}_{i=0}^\infty \subset \SD \setminus \Omega_h \nn\\
			&\text{such that }\lim_{i \to \infty} x_i = x \nn\\&\text{and} \lim_{i \to \infty} \nabla h(x_i) = v \Big\}. \nn
		\end{align}
	\end{definition}
	
	The following result specializes the Clarke subdifferential to the pointwise minimum of finitely many smooth functions; it follows from \cite[Prop.\ 2.3.12]{clarke1990optimization} applied to $-h = \max_{i} (-h_i)$.
	
	\begin{lemma}
		\label{lem:Clarke_min}
		Let $h_1, \ldots, h_p \colon \SD \to \BBR$ be continuously differentiable, and let $h(x) \triangleq \min_{i \in \{1,\ldots,p\}} h_i(x)$. For all $x \in \SD$, the \textit{active set} is defined as
		\begin{equation}
			\SA(x) \triangleq \big\{i \in \{1,\ldots,p\} \colon h_i(x) = h(x)\big\}. \label{eq:active_set}
		\end{equation}
		Then, for all $x \in \SD$,
		\begin{equation}
			\partial h(x) \subseteq \Co\{\nabla h_i(x) \colon i \in \SA(x)\}. \label{eq:Clarke_min_bound}
		\end{equation}
		Furthermore, since each $-h_i$ is continuously differentiable and therefore regular in the sense of \cite[Def.\ 2.3.4]{clarke1990optimization}, equality holds in \eqref{eq:Clarke_min_bound}.
	\end{lemma}

	The right-hand side of \eqref{eq:Clarke_min_bound} is the polytope generated by the active gradients $\{\nabla h_i(x)\}_{i \in \SA(x)}$; its vertices form a subset of these generators, with every generator a vertex provided they are affinely independent. The main results of this paper use the fact that the dual-algebraic evaluation deterministically returns one such generator.

	\begin{definition}
		\label{def:nonsmooth_CBF}
		Let $h \colon \SD \to \BBR$ be locally Lipschitz. The function $h$ satisfies the \textit{non-smooth CBF condition} on $\SC$ if there exists $\alpha \in \SK_e$ such that, for all $x \in \SC$,
		\begin{equation}
			\sup_{u \in \BBR^m} \inf_{\zeta \in \partial h(x)} \Big(\, \zeta\,(f(x) + G(x) u) + \alpha(h(x)) \,\Big) \geq 0.
			\label{eq:nonsmooth_CBF}
		\end{equation}
	\end{definition}
	
	\begin{remark}\rm
		\label{rem:active_form}
		If $h(x) = \min_i h_i(x)$ with each $h_i$ continuously differentiable, then the equality in Lemma~\ref{lem:Clarke_min} yields that \eqref{eq:nonsmooth_CBF} is equivalent to the requirement that, for all $i \in \SA(x)$,
		\begin{equation}
			\sup_{u \in \BBR^m} \Big(\, \Lf{f}{h_i}(x) + \Lf{G}{h_i}(x)\, u + \alpha(h(x)) \,\Big) \geq 0. \nn
		\end{equation}
		The simultaneous enforcement formulation in Section~\ref{sec:QP} realizes this stronger condition. \exampletriangle
	\end{remark}

	\section{Problem Formulation}
	\label{sec:problem}
	
	Consider the control-affine system \eqref{eq:system} with safe set $\SC = \{x \in \SD \colon h(x) \geq 0\}$ generated by the composite barrier
	\begin{equation}
		h(x) = \min_{i \in \{1,\ldots,p\}} h_i(x), \label{eq:problem_h}
	\end{equation}
	where each $h_i \colon \SD \to \BBR$ is continuously differentiable. As noted in Section~\ref{sec:intro}, $h$ is locally Lipschitz but non-differentiable on the locus where two or more constraints attain the minimum. Thus, $\nabla h$ is undefined on that locus and Fact~\ref{fact:CBF_invariance} does not apply directly.
	
	Let $C_h$ denote the number of elementary arithmetic operations required to evaluate $h$ in $\BBR$. We seek a state-feedback safety filter $u_* \colon \SD \to \BBR^m$ and a gradient-extraction procedure satisfying, for all $x \in \SD$, three objectives:
	
	\begin{enumerate}[label=(P\arabic*),itemindent=1pt,leftmargin=*]
		\item \label{P:invariance} \textit{Forward invariance.} The set $\SC$ is forward invariant with respect to the closed-loop system $\dot x = f(x) + G(x)\, u_*(x)$.
		\item \label{P:exactness} \textit{Exactness.} The gradient information supplied to the filter is an exact element of the Clarke subdifferential $\partial h(x)$, not a smooth surrogate or a stochastic estimate.
		\item \label{P:determinism} \textit{Deterministic bounded cost.} The filter data are produced by a fixed straight-line program whose operation count is $O(C_h)$, independent of the constraint count $p$, the state dimension $n$, the active-set cardinality $|\SA(x)|$, and the local geometry of $\bd\SC$. In particular, the procedure performs no convex-hull construction, no integer search, and no smoothing-parameter selection.
	\end{enumerate}
	
	Objective \ref{P:invariance} is the safety requirement; \ref{P:exactness} rules out the bias of smooth approximations and the non-determinism of stochastic methods; and \ref{P:determinism} rules out the state-dependent execution times of convex-hull and mixed-integer formulations, which is the property required for hard real-time certification. The remainder of the paper constructs a procedure that meets \ref{P:invariance}--\ref{P:determinism} simultaneously: the dual-algebraic evaluation supplies an exact vertex of $\partial h(x)$ at $O(C_h)$ cost, settling \ref{P:exactness} and \ref{P:determinism} (Sections~\ref{sec:dual_algebra}--\ref{sec:main_results}); simultaneous enforcement of the $\delta$-active constraints then yields \ref{P:invariance} (Theorem~\ref{thm:forward_invariance}).

	\section{Dual-Number Algebra}
	\label{sec:dual_algebra}
	
	This section develops the dual-number ring, the dual extension of a smooth function, and the directional-derivative identity that drives the rest of the paper.
	
	\subsection{The Dual-Number Ring}
	
	\begin{definition}
		\label{def:dual_numbers}
		The \textit{dual-number ring} is the quotient $\BBD \triangleq \BBR[\epsilon]/(\epsilon^2)$. Each element $z \in \BBD$ has a unique representation $z = a + b \epsilon$, where $a, b \in \BBR$ and $\epsilon \neq 0$ satisfies $\epsilon^2 = 0$. The \textit{real} and \textit{dual} parts of $z$ are denoted
		\begin{equation}
			\Re(z) \triangleq a, \qquad \Du(z) \triangleq b. \nn
		\end{equation}
	\end{definition}
	
	\begin{definition}
		\label{def:dual_arith}
		For $z_1 = a_1 + b_1 \epsilon$ and $z_2 = a_2 + b_2 \epsilon$ in $\BBD$,
		\begin{align}
			z_1 + z_2 &\triangleq (a_1 + a_2) + (b_1 + b_2) \epsilon, \nn \\
			z_1 \cdot z_2 &\triangleq a_1 a_2 + (a_1 b_2 + a_2 b_1) \epsilon. \nn
		\end{align}
		The constant $r \in \BBR$ is identified with $r + 0 \epsilon \in \BBD$.
	\end{definition}
	
	\begin{proposition}
		\label{prop:ring}
		The triple $(\BBD, +, \cdot)$ is a commutative ring with identity $1$. An element $z = a + b \epsilon$ is a unit if and only if $a \neq 0$, in which case
		\begin{equation}
			z^{-1} = \tfrac{1}{a} - \tfrac{b}{a^2}\, \epsilon. \nn
		\end{equation}
		The ideal $(\epsilon) = \{b\epsilon \colon b \in \BBR\}$ is the unique maximal ideal of $\BBD$.
	\end{proposition}
	
	\begin{proof}
		Associativity, commutativity, and distributivity are inherited from $\BBR[\epsilon]$. The expression for $z^{-1}$ is verified by direct computation using $\epsilon^2 = 0$. A nonzero element $b \epsilon$ is nilpotent and therefore not a unit, which confirms that $(\epsilon)$ is the unique maximal ideal.
	\end{proof}

	\subsection{Dual Extension of Smooth Functions}
	
	\begin{definition}
		\label{def:dual_ext}
		Let $\phi \colon \BBR \to \BBR$ be continuously differentiable. The \textit{dual extension} of $\phi$ is the function $\tilde{\phi} \colon \BBD \to \BBD$ defined by
		\begin{equation}
			\tilde{\phi}(a + b \epsilon) \triangleq \phi(a) + \phi'(a)\, b\, \epsilon.
			\label{eq:dual_ext_scalar}
		\end{equation}
	\end{definition}
	
	This analytic definition directly embeds the algebraic structure of the dual ring. Applying the extension formula \eqref{eq:dual_ext_scalar} to the elementary real functions $\phi(x) = x+r$, $\phi(x) = rx$, $\phi(x) = x^2$, and $\phi(x) = x^{-1}$ exactly reproduces the dual arithmetic operations and the unit inverse formalized in Definition~\ref{def:dual_arith} and Proposition~\ref{prop:ring}. The following lemma guarantees that this mapping remains closed under function composition.

	\begin{lemma}
		\label{lem:composition}
		Let $\phi, \psi \colon \BBR \to \BBR$ be continuously differentiable, and let $\tilde{\phi}, \tilde{\psi}$ denote their dual extensions. Then $\widetilde{\phi \circ \psi} = \tilde{\phi} \circ \tilde{\psi}$.
	\end{lemma}
	
	\begin{proof}
		Let $z = a + b\epsilon \in \BBD$. Then,
		\begin{equation}
			\tilde{\psi}(z) = \psi(a) + \psi'(a)\, b\, \epsilon. \label{eq:proof_comp_inner}
		\end{equation}
		Substituting \eqref{eq:proof_comp_inner} into $\tilde{\phi}$ and using $\epsilon^2 = 0$,
		\begin{align}
			\tilde{\phi}(\tilde{\psi}(z)) &= \phi(\psi(a)) + \phi'(\psi(a))\, \psi'(a)\, b\, \epsilon, \label{eq:proof_comp_chain}\nn\\
			&= (\phi \circ \psi)(a) + (\phi \circ \psi)'(a)\, b\, \epsilon \nn \\
			&= \widetilde{\phi \circ \psi}(z), \nn
		\end{align}
		where the second equality follows from the classical chain rule applied to the dual part.
	\end{proof}
	
	\begin{remark}\rm
		\label{rem:multivar}
		Lemma~\ref{lem:composition} extends componentwise to vector-valued maps. If $\eta\colon \BBR^n \to \BBR^k$ is continuously differentiable, its \textit{dual extension} is
		\begin{equation}
			\tilde{\eta}(x + v \epsilon) \triangleq \eta(x) + \eta'(x)\, v\, \epsilon \in \BBD^k, \label{eq:dual_ext_vec}
		\end{equation}
		where  $v \in \BBR^n$. The multivariate chain rule guarantees that compositions of dual-extended maps remain dual extensions; in particular, the dual evaluation of each straight-line program that computes $\eta$ over $\BBR$ becomes a straight-line program that computes $\tilde{\eta}$ over $\BBD$ when every elementary operation is replaced by its dual counterpart. \exampletriangle
	\end{remark}
	
	The following theorem makes the connection to Lie derivatives explicit and is the foundation for the rest of the paper.
	
	\begin{theorem}
		\label{thm:Lie_extraction}
		Let $h \colon \SD \to \BBR$ be continuously differentiable, and let $v \colon \SD \to \BBR^n$ be continuous. Define the \textit{dual lift} of $x$ along $v$ as
		\begin{equation}
			x_{\BBD,v} \triangleq x + v(x)\, \epsilon \in \BBD^n. \label{eq:dual_state}
		\end{equation}
		Then, 
		\begin{equation}
			\tilde{h}\bigl(x_{\BBD,v}\bigr) = h(x) + \Lf{v}{h}(x)\, \epsilon.
			\label{eq:Lie_ext}
		\end{equation}
	\end{theorem}
	
	\begin{proof}
		Specializing \eqref{eq:dual_ext_vec} to $\eta = h$ and to the dual lift \eqref{eq:dual_state} yields
		\begin{equation}
			\tilde{h}(x + v(x) \epsilon) = h(x) + \nabla h(x)\, v(x)\, \epsilon \nn,
		\end{equation}
		and the definition of the Lie derivative confirms \eqref{eq:Lie_ext}.
	\end{proof}
	
	\begin{remark}\rm
		\label{rem:control_seed}
		For the control-affine system \eqref{eq:system}, the full dual lift along the closed-loop vector field is $x_{\BBD,f + G u} \isdef x + (f(x) + G(x) u)\, \epsilon$, and Theorem~\ref{thm:Lie_extraction} implies
		\begin{equation}
			\tilde{h}\bigl(x_{\BBD,f + G u}\bigr) = h(x) + \bigl(\,\Lf{f}{h}(x) + \Lf{G}{h}(x)\, u\,\bigr)\, \epsilon. \nn
		\end{equation}
		Since the CBF inequality is affine in $u$, however, it is preferable in practice to evaluate $\tilde h$ at the $m+1$ seeds, namely, $x_{\BBD,f}$ and $x_{\BBD,\col_j(G)}$ for $j \in \{1,\ldots,m\}$, and then to recover the linear function of $u$ from the dual parts. \exampletriangle
	\end{remark}

	\section{Main Results: Vertex Selection and Forward Invariance}
	\label{sec:main_results}
	
	This section contains the main theoretical contributions. We define the dual-algebraic evaluation of a pointwise minimum, prove that it selects a vertex of the Clarke subdifferential, and establish forward invariance of the safe set under the resulting simultaneous-enforcement safety filter.
	
	\subsection{Lexicographic Comparison on $\BBD$}
	
	\begin{definition}
		\label{def:lex_order}
		Let  $z_1 = a_1 + b_1 \epsilon\in\BBD$ and $z_2 = a_2 + b_2 \epsilon\in\BBD$. We write $z_1 \preceq_L z_2$ if and only if one of the following conditions hold:
		
		\begin{enumerate}
			\item $a_1 < a_2.$
			\item $a_1 = a_2 \ \text{and}\ b_1 \leq b_2$.
		\end{enumerate} 
		The \textit{lexicographic minimum} on $\BBD$ is denoted
		\begin{equation}
			\lexmin(z_1, z_2) \triangleq \begin{cases} z_1, & z_1 \preceq_L z_2, \\ z_2, & \text{otherwise.} \end{cases} \nn
		\end{equation}
		The \textit{real-part minimum} on $\BBD$ is
		\begin{equation}
			\minRe(z_1, z_2) \triangleq \begin{cases} z_1, & \Re(z_1) \leq \Re(z_2), \\ z_2, & \text{otherwise.} \end{cases} \label{eq:min_real}
		\end{equation}
	\end{definition}
	
	The two operations agree whenever $\Re(z_1) \neq \Re(z_2)$. At a real-part tie, $\minRe$ resolves the tie by the order of the operands (i.e., left-to-right precedence under the inclusive inequality), whereas $\lexmin$ resolves it by the dual parts. Equation \eqref{eq:min_real} is precisely what is realized by a standard IEEE-754 floating-point comparison of dual numbers stored as pairs $(a,b)$: the comparison instruction examines $a$ only, and the result is deterministic at every state. For this reason, $\minRe$ is the default in the remainder of the paper. The variant $\lexmin$ is included for completeness and is used in some implementations where deterministic tie-breaking by the dual part is desirable.
	
	\begin{definition}
		\label{def:dual_eval}
		Let $h_1, \ldots, h_p \colon \SD \to \BBR$ be continuously differentiable and let $h(x) \triangleq \min_{i\in\{1,\ldots,p\}} h_i(x)$. The \textit{dual-algebraic evaluation} of $h$ at the lifted state $x_{\BBD,v}$ is
		\begin{equation}
			\tilde{h}\bigl(x_{\BBD,v}\bigr) \triangleq \minRe\!\left(\tilde{h}_1\bigl(x_{\BBD,v}\bigr),\, \ldots,\, \tilde{h}_p\bigl(x_{\BBD,v}\bigr)\right), \label{eq:dual_eval}
		\end{equation}
		where $\minRe$ is applied left-to-right as a binary operation.
	\end{definition}
	
	\subsection{Vertex Selection}
	
	The following lemma is the core algebraic identity. It shows that the dual-algebraic evaluation simultaneously returns the value $h(x)$ and the Lie derivative of \emph{one} active constraint, deterministically selected by the order of the operands.
	
	For all $x\in\SD$, define
	\begin{equation}
		i_*(x) \triangleq \min\, \SA(x).
		\label{eq:kstar}
	\end{equation}
	
	\begin{lemma}
		\label{lem:dual_eval_structure}
		Let $x \in \SD$,  $v \colon \SD \to \BBR^n$ be continuous,  $h_1, \ldots, h_p \colon \SD \to \BBR$ be continuously differentiable, and define $h(x) \triangleq \min_{i\in\{1,\ldots,p\}} h_i(x)$. Then,
		\begin{equation}
			\tilde{h}\bigl(x_{\BBD,v}\bigr) = h(x) + \Lf{v}{h_{i_*(x)}}(x)\, \epsilon. \label{eq:dual_eval_result}
		\end{equation}
	\end{lemma}
	
	\begin{proof}
		Theorem~\ref{thm:Lie_extraction} implies that, for all $i \in \{1,\ldots,p\}$,
		\begin{equation}
			\tilde{h}_i\bigl(x_{\BBD,v}\bigr) = h_i(x) + \Lf{v}{h_i}(x)\, \epsilon. \label{eq:proof_dual_eval_each}
		\end{equation}
		Using Definition~\ref{def:lex_order},  the left-to-right reduction in \eqref{eq:dual_eval}  yields, by induction on $p$, the operand $\tilde{h}_{\bar i_*}(x_{\BBD,v})$ with
		\begin{equation}
			\bar i_* = \min\biggl\{\, i \in \{1,\ldots,p\} \,\colon\, h_i(x) = \min_{j\in \{1,\ldots,p\}} h_j(x) \,\biggr\}. \nn
		\end{equation}
		By \eqref{eq:active_set} and \eqref{eq:kstar}, $\bar i_* = i_*(x)$. Substituting $i = i_*(x)$ in \eqref{eq:proof_dual_eval_each} and using $h_{i_*(x)}(x) = h(x)$ yields \eqref{eq:dual_eval_result}.
	\end{proof}
	
	The next theorem is the first main result. It establishes that the gradient extracted by the dual-algebraic evaluation is a \emph{vertex} of the Clarke subdifferential and is therefore an admissible subgradient at every $x \in \SD$.
	
	\begin{theorem}
		\label{thm:vertex_selection}
		Let $x \in \SD$,  $v \colon \SD \to \BBR^n$ be continuous,  $h_1, \ldots, h_p \colon \SD \to \BBR$ be continuously differentiable, and define $h(x) \triangleq \min_{i\in\{1,\ldots,p\}} h_i(x)$. Then, the dual-algebraic evaluation \eqref{eq:dual_eval} satisfies
		\begin{equation}
			\Re\bigl(\tilde{h}\bigl(x_{\BBD,v}\bigr)\bigr) = h(x), \label{eq:thm1_real}
		\end{equation}
		and the corresponding gradient $\nabla h_{i_*(x)}(x)$ is a generator of the polytope $\Co\{\nabla h_i(x) \colon i \in \SA(x)\}$. In particular,
		\begin{equation}
			\nabla h_{i_*(x)}(x) \in \partial h(x). \label{eq:thm1_subgradient}
		\end{equation}
		If the active gradients $\{\nabla h_i(x)\}_{i \in \SA(x)}$ are affinely independent, then this generator is a vertex (i.e., extreme point) of the polytope. Moreover, the map $x \mapsto i_*(x)$ is a deterministic function of the real-part values $h_1(x), \ldots, h_p(x)$ alone.
	\end{theorem}
	
	\begin{proof}
		Lemma~\ref{lem:dual_eval_structure} combined with the active-set equality $h_{i_*(x)}(x) = h(x)$ directly yields \eqref{eq:thm1_real}.
		
		Since the selected index $i_*(x)$ belongs to the active set $\SA(x)$, the gradient $\nabla h_{i_*(x)}(x)$ is a generator of the polytope $\Co\{\nabla h_i(x) \colon i \in \SA(x)\}$. By the equality in Lemma~\ref{lem:Clarke_min}, this polytope equals $\partial h(x)$, verifying \eqref{eq:thm1_subgradient}.
		Moreover, $\nabla h_{i_*(x)}(x)$ is one of the finitely many generators of $\Co\{\nabla h_i(x) \colon i \in \SA(x)\}$; when these generators are affinely independent, each is an extreme point, thus $\nabla h_{i_*(x)}(x)$ is then a vertex of the polytope.
		
		Equation \eqref{eq:kstar} isolates the selection index $i_*(x)$ using exclusively the real-valued constraint magnitudes $h_1(x), \ldots, h_p(x)$, establishing the deterministic map.
	\end{proof}
	
	\begin{remark}\rm
		\label{rem:vertex_geometry}
		Theorem~\ref{thm:vertex_selection} formalizes the central feature of the dual-algebraic approach: the evaluation does not approximate the Clarke generalized gradient and does not enumerate the polytope $\partial h(x)$, but instead returns exactly one generator of this polytope at $O(1)$ overhead per dual operation. Methods that compute the full Clarke subdifferential incur $O(p^{\lfloor n/2 \rfloor})$ convex-hull cost \cite{barber1996quickhull} to recover the entire polytope. \exampletriangle
	\end{remark}

	\subsection{Forward Invariance via Simultaneous Enforcement}
	
	The vertex selected by the dual-algebraic evaluation is admissible as a Clarke subgradient, but enforcing the CBF inequality at this vertex alone is, in general, insufficient for forward invariance of $\SC$. The following counterexample makes this explicit.
	
	\begin{example}
		\label{ex:counter}
		\rm
		Consider a single-integrator $\dot{x} = u$ on $\BBR^2$ with $u \in \BBR^2$, and let
		\begin{equation}
			h_1(x) = x_{(1)}, \quad h_2(x) = x_{(2)}, \quad h(x) = \min\{h_1(x), h_2(x)\}. \nn
		\end{equation}
		Note that, for $x = (0,0) \in \bd \SC$, $\SA(x) = \{1, 2\}$ and $i_*(x) = 1$. The dual-algebraic evaluation returns the vertex $\nabla h_1(0) = [1~ 0]$, and the single-vertex CBF inequality at $u$ reduces to $u_{(1)} \geq -\alpha(0) = 0$. The component $u_{(2)}$ is unconstrained, so the nominal choice $u = (0,-1)$ trivially satisfies the inequality. Under this control, however, $\dot{x}_{(2)} = -1$, and, for all $t > 0$, $h(x(t)) = -t < 0$, which violates $\SC$ at $t = 0^+$. Figure~\ref{fig:counter} illustrates the geometry.
		
		\begin{figure}[ht!]
			\centering
			\begin{tikzpicture}[>=Latex, scale=1.35, font=\scriptsize]
				\fill[blue!8] (0,0) rectangle (2.3,2.0);
				\node[blue!50!black, anchor=north east] at (2.25,1.95) {$\SC$};
				\fill[red!7] (-0.8,-1.4) rectangle (2.5,0);
				\node[red!55!black, anchor=south east] at (2.45,-1.35) {$\{h_2 < 0\}$};
				\draw[blue!60!black, very thick] (0,2.2) -- (0,0) -- (2.5,0);
				\draw[->, black!60] (-0.95,0) -- (2.75,0) node[right, font=\scriptsize] {$x_{(1)}$};
				\draw[->, black!60] (0,-1.6) -- (0,2.55) node[above, font=\scriptsize] {$x_{(2)}$};
				\node[blue!60!black, anchor=north] at (1.85,-0.05) {$h_1=0$};
				\node[blue!60!black, anchor=south east, inner sep=2pt] at (-0.05,1.55) {$h_2=0$};
				\fill[black] (0,0) circle (1.4pt);
				\node[anchor=north east, inner sep=2pt] at (-0.05,-0.05) {$x_0 = 0$};
				\draw[->, very thick, green!50!black] (0,0) -- (1.0,0);
				\node[green!45!black, anchor=south]   at (0.55,0.05) {$\nabla h_1$};
				\draw[->, very thick, green!50!black] (0,0) -- (0,1.0);
				\node[green!45!black, anchor=west]    at (0.05,0.55) {$\nabla h_2$};
				\node[green!45!black, draw=green!45!black, fill=white,
				inner sep=2pt, rounded corners=2pt, anchor=south west] (sel)
				at (1.25,1.25) {selected: $\nabla h_{i_*} = \nabla h_1$};
				\draw[->, green!45!black, dashed, shorten >=2pt]
				(sel.south west) to[bend right=19] (.85,0.04);
				\draw[->, very thick, red] (0,0) -- (0,-1.25);
				\node[red, anchor=west, inner sep=2pt] at (0.10,-0.55) {$\dot{x} = u_{\rm n} = (0,-1)$};
				\node[red, anchor=west, inner sep=2pt] at (0.10,-1.05) {exits $\SC$ for $t>0$};
			\end{tikzpicture}
			\caption{Failure of single-vertex enforcement at the corner $x_0 = (0,0)$ of the safe set $\SC = \{x \colon h_1(x) \geq 0,\ h_2(x) \geq 0\}$. The dual-algebraic evaluation selects $i_*(x_0) = 1$ and returns the vertex $\nabla h_1(0) = [1,0]$ (green). The single-vertex CBF inequality $\nabla h_1(0)\cdot u \geq 0$ enforces $u_{(1)} \geq 0$ but leaves $u_{(2)}$ unconstrained, so the nominal $u_{\rm n} = (0,-1)$ is admissible and drives $x$ into the unsafe region $\{h_2 < 0\}$ (red).}
			\label{fig:counter}
		\end{figure}
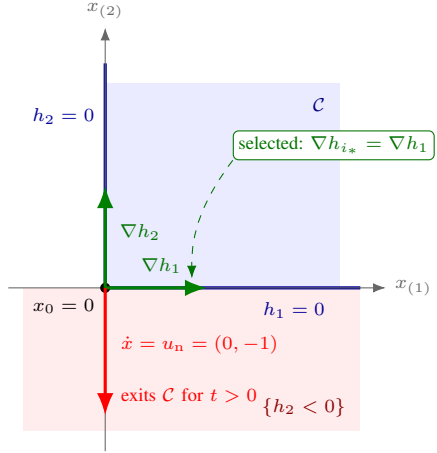
		
		The conclusion is that a Clarke subgradient suffices for a \emph{descent} direction in nonsmooth optimization but does not, on its own, certify forward invariance: at a non-smooth boundary point of $\SC$, all active constraints must be enforced simultaneously, not just one vertex of the active polytope. \exampletriangle
	\end{example}
	
	We now formulate the QP that simultaneously enforces all $\delta$-active constraints with Lie derivatives obtained from the dual-algebraic evaluation, and establish forward invariance.
	
	\begin{definition}
		\label{def:delta_active}
		For $\delta \geq 0$ and $x \in \SD$, the \textit{$\delta$-active set} is
		\begin{equation}
			\SA_\delta(x) \triangleq \{i \in \{1,\ldots,p\} \colon h_i(x) \leq h(x) + \delta\}.\nn
		\end{equation}
	\end{definition}
	
	Note that $\SA_0(x) = \SA(x)$, and that $\SA_\delta(x)$ parallels the almost-active sets of \cite{glotfelter2019hybrid}. The set $\SA_\delta(x)$ can be read off from the dual-algebraic evaluation by inspecting the real parts of $\tilde{h}_1(x_{\BBD,v}),\ldots, \tilde{h}_p(x_{\BBD,v})$ at no additional cost.
	
	\begin{assumption}
		\label{assum:CBF}
		For all $x \in \SD$ and for all $i \in \SA(x) \cap \{i \colon h_i(x) = 0\}$, the row vector $\Lf{G}{h_i}(x) \in \BBR^{1 \times m}$ is nonzero, and there exists $u \in \BBR^m$ such that, for all $i \in \SA(x)$,
		\begin{equation}
			\Lf{f}{h_i}(x) + \Lf{G}{h_i}(x)\, u + \alpha(h_i(x)) \geq 0. \nn
		\end{equation}
	\end{assumption}
	
	Assumption~\ref{assum:CBF} is a joint controllability-feasibility hypothesis for non-smooth CBFs; it parallels the strict-feasibility condition $w_*(x) > 0$ of \cite[Prop.\ 3]{glotfelter2017nonsmooth}, which quantifies the width of the feasible set of the multi-constraint QP. It is implied, for example, by transversality of the active gradients together with controllability of each smooth constraint.
	
	\begin{theorem}
		\label{thm:forward_invariance}
		Let $h_1,\ldots,h_p \colon \SD \to \BBR$ be continuously differentiable, and define $h(x) \triangleq \min_{i\in \{1,\ldots,p\}} h_i(x)$. Suppose Assumption~\ref{assum:CBF} holds, and let $\alpha \in \SK_e$. Let $\delta > 0$, and let $u_* \colon \SD \to \BBR^m$ be locally Lipschitz such that, for all $x \in \SD$ and all $i \in \SA_\delta(x)$,
		\begin{equation}
			\Lf{f}{h_i}(x) + \Lf{G}{h_i}(x)\, u_*(x) + \alpha(h_i(x)) \geq 0, \label{eq:simult_CBF}
		\end{equation}
		where the Lie derivatives are obtained from the dual-algebraic evaluation,
		\begin{align}
			\Lf{f}{h_i}(x) &= \Du\bigl(\tilde{h}_i(x_{\BBD,f})\bigr), \nn\\ \Lf{\col_j(G)}{h_i}(x) &= \Du\bigl(\tilde{h}_i(x_{\BBD,\col_j(G)})\bigr). \nn
		\end{align}
		Then, $\SC$ is forward invariant with respect to the closed-loop system $\dot{x} = f(x) + G(x)\, u_*(x)$.
	\end{theorem}
	
	\begin{proof}
		Let $x\colon[0,\infty)\to\BBR^n$ denote the maximal solution of $\dot{x} = f(x) + G(x)\, u_*(x)$ from the initial condition $x(0) \in \SC$.
		Assume, for contradiction, that the trajectory escapes $\SC$ in finite time, which implies
		\begin{equation}
			\bar t \triangleq \inf\{\, t > 0 \colon h(x(t)) < 0 \,\} \label{eq:proof_inv_tbar}
		\end{equation}
		exists. 
		Continuity of $h \circ x$ together with $h(x(0)) \geq 0$ implies that $\bar t \in (0, \infty)$, $h(x(\bar t)) = 0$, and, for all $t \in [0, \bar t]$, $h(x(t)) \geq 0$.
		
		By the definition \eqref{eq:proof_inv_tbar}, there exists a sequence $\{s_j\}_{j=0}^\infty$ such that as $j\to\infty$, $s_j \downarrow 0$, and, for all $j\in\BBN$, $h(x(\bar t + s_j)) < 0$. For each $j\in\BBN$, the strict inequality is witnessed by some index $i_j \in \{1,\ldots,p\}$ with $h_{i_j}(x(\bar t + s_j)) < 0$. Since the index set is finite, the pigeonhole principle yields that there exists $i_\dagger \in \{1,\ldots,p\}$ and a subsequence $\{\bar s_j\}_{j=0}^\infty$ such that as $j\to\infty$, $\bar s_j \downarrow 0$, and, for all $j\in\BBN$,
		\begin{equation}
			h_{i_\dagger}\bigl(x(\bar t + \bar s_j)\bigr) < 0. \label{eq:proof_inv_witness}
		\end{equation}
		Continuity of $h_{i_\dagger} \circ x$ at $\bar t$ gives $h_{i_\dagger}(x(\bar t)) \leq 0$. Since, in addition, $h_{i_\dagger}(x(\bar t)) \geq h(x(\bar t)) = 0$ by definition of the pointwise minimum, it follows that
		\begin{equation}
			h_{i_\dagger}(x(\bar t)) = 0, \quad i_\dagger \in \SA(x(\bar t)) \subseteq \SA_\delta(x(\bar t)).\nn 
		\end{equation}
		
		Continuity of $(h_{i_\dagger} \circ x) - (h \circ x)$ at $\bar t$, combined with the strict inequality $0 < \delta$, yields that there exists $\eta > 0$ such that, for all $t \in [\bar t, \bar t + \eta]$,
		\begin{equation}
			h_{i_\dagger}(x(t)) - h(x(t)) < \delta,  \nn
		\end{equation}
		which implies that, for all $t \in [\bar t, \bar t + \eta]$,  $i_\dagger \in \SA_\delta(x(t))$; shrinking $\eta$ if necessary, $h_{i_\dagger}(x(t))$ remains in the domain of $\alpha$ for all $t \in [\bar t, \bar t + \eta]$.
		Thus, \eqref{eq:simult_CBF} implies that, for all $t \in [\bar t, \bar t + \eta]$,
		\begin{equation}
			\frac{\rmd}{\rmd t} h_{i_\dagger}(x(t)) \geq -\alpha\bigl(h_{i_\dagger}(x(t))\bigr). \label{eq:proof_inv_diffineq}
		\end{equation}
		Since $\alpha \in \SK_e$ with $\alpha(0) = 0$, the comparison equation $\frac{\rmd}{\rmd t} W = -\alpha(W)$ with $W(\bar t) = 0$ has unique solution $W \equiv 0$. Applying the comparison lemma \cite[Lem.\ 3.4]{khalil2002nonlinear} to \eqref{eq:proof_inv_diffineq} with $h_{i_\dagger}(x(\bar t)) = 0 = W(\bar t)$ yields that, for all $ t \in [\bar t, \bar t + \eta],$
		\begin{equation}
			h_{i_\dagger}(x(t)) \geq W(t) = 0, \nn
		\end{equation}
		which,  for all $\bar s_j \in (0, \eta)$, contradicts \eqref{eq:proof_inv_witness}. Hence no such escape time $\bar t$ exists, and $\SC$ is forward invariant with respect to the closed-loop system.
	\end{proof}
	
	\begin{remark}\rm
		\label{rem:computational_determinism}
		Theorem~\ref{thm:forward_invariance} exposes the central computational requirement of non-smooth safety: certifying forward invariance at a non-smooth boundary point requires simultaneous enforcement of every $\delta$-active constraint, not just one. Existing approaches meet this requirement in different ways: \cite{glotfelter2017nonsmooth} enforces one constraint per component function at every state, \cite{chen2018obstacle} selects constraints through binary variables in a mixed-integer QP, and explicit construction of $\partial h(x)$ entails a convex-hull computation \cite{barber1996quickhull}; the costs of the latter two depend on the local geometry of $\bd \SC$. The dual-algebraic evaluation circumvents this bottleneck: by lifting the state along $f$ and along each column $\col_j(G)$ and evaluating $\tilde h_i$ on the unmodified straight-line program of $h_i$, the hardware extracts both the barrier value $h_i(x)$ and its Lie derivatives in a single forward pass, supplying the safety-filter QP \eqref{eq:simult_CBF} with the exact gradient information needed at every state. As will be stated in Proposition~\ref{prop:complexity}, the worst-case execution time is a constant multiple of the cost of evaluating $h$ in $\BBR$, independent of $|\SA(x)|$ and of $\bd \SC$. \exampletriangle
	\end{remark}

	\begin{remark}\rm
		\label{rem:delta_choice}
		In implementation, the parameter $\delta > 0$ in \eqref{eq:simult_CBF} provides a margin within which constraints are jointly enforced. Smaller $\delta$ reduces the average number of active constraints in the QP, at the cost of requiring more responsive constraint switching as the trajectory approaches the non-smooth locus; in the limit $\delta \to 0^+$, only the constraints exactly attaining $h(x)$ are enforced, recovering the strict $\SA(x)$-based formulation. \exampletriangle
	\end{remark}

	\section{Extensions}
	\label{sec:extensions}
	
	This section extends the dual-algebraic framework in two directions: barrier functions formed by arbitrary finite compositions of pointwise $\min$ and $\max$ operations, and systems of relative degree greater than one.
	
	\subsection{Compositions of Pointwise $\min$ and $\max$}
	
	Practical safety specifications often combine constraints via both unions and intersections, that is, by Boolean compositions of barrier functions \cite{glotfelter2017nonsmooth,glotfelter2019hybrid}. We define the class of finite composite functions $\mathcal{F}_p$ mapping $\BBR\times\ldots\times \BBR$ to $\mathbb{R}$ inductively. For every index $i \in \{1, \ldots, p\}$, the coordinate projection $\pi_i(y) \triangleq y_{(i)}$ belongs to $\mathcal{F}_p$. If $\Phi_1, \Phi_2 \in \mathcal{F}_p$, then their pointwise minimum $\min(\Phi_1, \Phi_2)$ and maximum $\max(\Phi_1, \Phi_2)$ also belong to $\mathcal{F}_p$. This generates barrier functions of the form
	\begin{equation}
		h(x) \triangleq \Phi(h_1(x), \ldots, h_p(x)),\nn
	\end{equation}
	where $\Phi \in \mathcal{F}_p$.
	
	The maximum of dual numbers is defined, symmetrically to \eqref{eq:min_real}, by
	\begin{equation}
		\maxRe(z_1, z_2) \triangleq \begin{cases} z_1, & \Re(z_1) \geq \Re(z_2), \\ z_2, & \text{otherwise.} \end{cases} \nn
	\end{equation}
	
	For each composite function $\Phi \in \mathcal{F}_p$, we define its dual-algebraic evaluation $\Phi_\BBD \colon \BBD^p \to \BBD$  by structural recursion. If $\Phi = \pi_i$, then, for all $z \in \BBD^p$, $\Phi_\BBD(z) \triangleq z_{(i)}$. If $\Phi = \min(\Phi_1, \Phi_2)$, then $\Phi_\BBD(z) \triangleq \minRe(\Phi_{1,\BBD}(z), \Phi_{2,\BBD}(z))$. If $\Phi = \max(\Phi_1, \Phi_2)$, then $\Phi_\BBD(z) \triangleq \maxRe(\Phi_{1,\BBD}(z), \Phi_{2,\BBD}(z))$. The following theorem generalizes Theorem~\ref{thm:vertex_selection} to these exact compositions; Theorem~\ref{thm:cnf_invariance} then extends the forward-invariance guarantee of Theorem~\ref{thm:forward_invariance} to this class.
	
	\begin{theorem}
		\label{thm:composition}
		Let $x \in \SD$, let $v \colon \SD \to \BBR^n$ be continuous, and let $h_1, \ldots, h_p \colon \SD \to \BBR$ be continuously differentiable. Let $\Phi \in \mathcal{F}_p$, and define $h(x) \triangleq \Phi(h_1(x), \ldots, h_p(x))$. Then, the dual-algebraic evaluation $\Phi_\BBD$ satisfies
		\begin{equation}
			\Phi_\BBD\bigl(\tilde h_1(x_{\BBD,v}), \ldots, \tilde h_p(x_{\BBD,v})\bigr) = h(x) + \Lf{v}{h_{i_*(x)}}(x)\, \epsilon, \label{eq:composition_result}
		\end{equation}
		where the index $i_*(x) \in \{1,\ldots,p\}$ is determined by the path through the recursive structure of $\Phi_\BBD$ that, at every $\minRe$ node, retains the operand with the smaller real part (breaking ties by left-to-right precedence), and at every $\maxRe$ node retains the operand with the larger real part. Moreover, with $\SA(x) \triangleq \{i \colon h_i(x) = h(x)\}$,
		\begin{equation}
			\nabla h_{i_*(x)}(x) \in \Co\{\nabla h_i(x) \colon i \in \SA(x)\} \supseteq \partial h(x), \label{eq:composition_subgrad}
		\end{equation}
		and, if $\Phi$ is composed of $\min$ operations only, or of $\max$ operations only, then $\nabla h_{i_*(x)}(x) \in \partial h(x)$.
	\end{theorem}
	
	\begin{proof}
		We argue by structural induction on the function class $\mathcal{F}_p$.
		
		For the base case, consider the case where there exists $i\in\{1,\ldots,p\}$ such that $\Phi = \pi_i$. It thus follows that the projection isolates a single smooth coordinate. Therefore, Theorem~\ref{thm:Lie_extraction} with $i_*(x) = i$ and Theorem~ \ref{thm:vertex_selection} confirm \eqref{eq:composition_result} and \eqref{eq:composition_subgrad}.
		
		For the inductive step, consider the case where $\Phi = \min(\Phi_1, \Phi_2)$, and where the result already holds for $\Phi_1$ and $\Phi_2$. Define the intermediate dual scalars
		\begin{align}
			\bar z_1 &\triangleq \Phi_{1,\BBD}\bigl(\tilde h_1(x_{\BBD,v}), \ldots, \tilde h_p(x_{\BBD,v})\bigr) \nn\\
			&= \Phi_1\bigl(h_1(x), \ldots, h_p(x)\bigr) + \Lf{v}{h_{i_{*,1}(x)}}(x)\, \epsilon, \label{eq:proof_comp_z1} \\
			\bar z_2 &\triangleq \Phi_{2,\BBD}\bigl(\tilde h_1(x_{\BBD,v}), \ldots, \tilde h_p(x_{\BBD,v})\bigr) \nn\\
			&= \Phi_2\bigl(h_1(x), \ldots, h_p(x)\bigr) + \Lf{v}{h_{i_{*,2}(x)}}(x)\, \epsilon, \label{eq:proof_comp_z2}
		\end{align}
		where \eqref{eq:proof_comp_z1} and \eqref{eq:proof_comp_z2} apply the inductive hypothesis, extracting the selected indices $i_{*,1}(x)$ and $i_{*,2}(x)$ from their respective subtrees. Evaluating $\minRe(\bar z_1, \bar z_2)$ according to \eqref{eq:min_real} selects $\bar z_1$ for the case where $\Phi_1 \leq \Phi_2$ and $\bar z_2$ otherwise, resolving ties left-to-right. This deterministic comparison yields
		\begin{equation}
			\Phi_\BBD\bigl(\{\tilde h_j(x_{\BBD,v})\}_{j=1}^p\bigr) = h(x) + \Lf{v}{h_{i_*(x)}}(x)\, \epsilon, \nn
		\end{equation}
		where $i_*(x) = i_{*,1}(x)$ for the case where $\Phi_1 \leq \Phi_2$, and $i_*(x) = i_{*,2}(x)$ otherwise. The symmetric logic holds identical for $\Phi = \max(\Phi_1, \Phi_2)$.
		
		To show \eqref{eq:composition_subgrad}, note that the value returned at every node of the recursion equals the value of one of its operands, so the index $i_*(x)$ produced by the recursion is active, that is, $h_{i_*(x)}(x) = h(x)$; the gradient $\nabla h_{i_*(x)}(x)$ is therefore a generator of $\Co\{\nabla h_i(x) \colon i \in \SA(x)\}$. The inclusion $\partial h(x) \subseteq \Co\{\nabla h_i(x) \colon i \in \SA(x)\}$ follows by structural induction on $\Phi$ from \cite[Prop.\ 2.3.12]{clarke1990optimization}, applied at each $\max$ node and, through the identity $\min(\Phi_1,\Phi_2) = -\max(-\Phi_1,-\Phi_2)$, at each $\min$ node. If $\Phi$ is composed of $\min$ operations only, then $h$ equals the pointwise minimum of the leaf functions and the equality in Lemma~\ref{lem:Clarke_min} yields $\Co\{\nabla h_i(x) \colon i \in \SA(x)\} = \partial h(x)$; if $\Phi$ is composed of $\max$ operations only, the symmetric argument applies, since each $h_i$ is continuously differentiable and therefore regular. In both cases, $\nabla h_{i_*(x)}(x) \in \partial h(x)$.
	\end{proof}
	
	\begin{remark}\rm
		\label{rem:mixed_composition}
		For compositions that contain both $\min$ and $\max$ operations, the inclusion in \eqref{eq:composition_subgrad} can be strict, and the extracted gradient need not belong to $\partial h(x)$. For $n = 1$, $\Phi = \max(\min(\pi_1,\pi_2),\pi_3)$ with $h_1(x) = x$, $h_2(x) = -x$, and $h_3(x) = 0$ gives $h \equiv 0$, so $\partial h(0) = \{0\}$, while the routing selects $\nabla h_{i_*(0)}(0) = \nabla h_1(0) = 1$. The forward-invariance guarantee of Theorem~\ref{thm:cnf_invariance} below relies on the activity of the routed leaf, not on membership of its gradient in $\partial h(x)$, and is unaffected. \exampletriangle
	\end{remark}
	
	\begin{remark}\rm
		\label{rem:tree_traversal}
		Theorem~\ref{thm:composition} shows that the dual evaluation traces the recursive structure of $\Phi$ exactly once. This depth-first traversal extracts an active-leaf gradient without additional computation. Let $C_h$ denote the cost of evaluating the composite function $h$ over $\BBR$. The dual evaluation requires $O(C_h)$ operations, which ties the execution time to the algebraic complexity of the barrier and avoids the combinatorial enumeration of active faces required by convex-hull methods. \exampletriangle
	\end{remark}
	
	Theorem~\ref{thm:composition} extracts an active-leaf gradient for every $\Phi \in \mathcal{F}_p$; Example~\ref{ex:counter} shows that enforcement of a single gradient does not certify forward invariance. We now extend Theorem~\ref{thm:forward_invariance} to Boolean compositions. Since pointwise $\min$ and $\max$ satisfy the distributive lattice identities, every $\Phi \in \mathcal{F}_p$ admits a representation as a minimum of clause-wise maxima. We therefore consider the barrier
	\begin{equation}
		h(x) \triangleq \min_{c \in \{1,\ldots,q\}} M_c(x), \qquad M_c(x) \triangleq \max_{i \in \mathcal{S}_c} h_i(x), \label{eq:cnf_h}
	\end{equation}
	where, for all $c \in \{1,\ldots,q\}$, $\mathcal{S}_c \subseteq \{1,\ldots,p\}$ is nonempty. Each clause value $M_c$ is locally Lipschitz, and the clause values are intermediates of the dual-algebraic evaluation of \eqref{eq:cnf_h}.
	
	\begin{definition}
		\label{def:clause_active}
		For $\delta \geq 0$ and $x \in \SD$, the \textit{$\delta$-active clause set} of \eqref{eq:cnf_h} is
		\begin{equation}
			\SA^{\rmc}_\delta(x) \triangleq \{c \in \{1,\ldots,q\} \colon M_c(x) \leq h(x) + \delta\}, \nn
		\end{equation}
		and, for all $c \in \{1,\ldots,q\}$, the \textit{routed leaf} $i_c(x) \in \mathcal{S}_c$ is the index selected by the $\maxRe$ reduction of clause $c$ in Theorem~\ref{thm:composition}, which satisfies $h_{i_c(x)}(x) = M_c(x)$.
	\end{definition}
	
	Enforcement at a $\max$ node differs from enforcement at a $\min$ node: a union of safe sets is forward invariant provided that one of its members is forward invariant, so it suffices to enforce, in each $\delta$-active clause, the leaf selected by the routing. The following assumption and theorem parallel Assumption~\ref{assum:CBF} and Theorem~\ref{thm:forward_invariance}.
	
	\begin{assumption}
		\label{assum:CBF_cnf}
		For all $x \in \SD$ and all $c \in \SA^{\rmc}_0(x)$ with $M_c(x) = 0$, the row vector $\Lf{G}{h_{i_c(x)}}(x) \in \BBR^{1\times m}$ is nonzero, and there exists $u \in \BBR^m$ such that, for all $c \in \SA^{\rmc}_0(x)$,
		\begin{equation}
			\Lf{f}{h_{i_c(x)}}(x) + \Lf{G}{h_{i_c(x)}}(x)\, u + \alpha\bigl(h_{i_c(x)}(x)\bigr) \geq 0. \nn
		\end{equation}
	\end{assumption}
	
	\begin{theorem}
		\label{thm:cnf_invariance}
		Let $h_1,\ldots,h_p \colon \SD \to \BBR$ be continuously differentiable, and let $h$ be given by \eqref{eq:cnf_h}. Suppose Assumption~\ref{assum:CBF_cnf} holds, and let $\alpha \in \SK_e$. Let $\delta > 0$, and let $u_* \colon \SD \to \BBR^m$ be locally Lipschitz such that, for all $x \in \SD$ and all $c \in \SA^{\rmc}_\delta(x)$,
		\begin{equation}
			\Lf{f}{h_{i_c(x)}}(x) + \Lf{G}{h_{i_c(x)}}(x)\, u_*(x) + \alpha\bigl(h_{i_c(x)}(x)\bigr) \geq 0, \label{eq:cnf_CBF}
		\end{equation}
		where the Lie derivatives are obtained from the dual-algebraic evaluation as in Theorem~\ref{thm:forward_invariance}. Then, $\SC$ is forward invariant with respect to the closed-loop system $\dot{x} = f(x) + G(x)\, u_*(x)$.
	\end{theorem}
	
	\begin{proof}
		Let $x \colon [0,\infty) \to \BBR^n$ denote the maximal solution of $\dot x = f(x) + G(x)\, u_*(x)$ from the initial condition $x(0) \in \SC$. Assume, for contradiction, that the trajectory escapes $\SC$ in finite time, which implies that $\bar t \triangleq \inf\{t > 0 \colon h(x(t)) < 0\}$ exists. Continuity of $h \circ x$ together with $h(x(0)) \geq 0$ implies that $\bar t \in (0,\infty)$, $h(x(\bar t)) = 0$, and, for all $t \in [0,\bar t]$, $h(x(t)) \geq 0$.
		
		By the definition of $\bar t$, there exists a sequence $\{s_j\}_{j=0}^\infty$ such that, as $j \to \infty$, $s_j \downarrow 0$, and, for all $j \in \BBN$, $h(x(\bar t + s_j)) < 0$. By \eqref{eq:cnf_h}, each strict inequality is witnessed by a clause $c_j \in \{1,\ldots,q\}$ with $M_{c_j}(x(\bar t + s_j)) < 0$. Since the clause set is finite, the pigeonhole principle yields that there exist $c_\dagger \in \{1,\ldots,q\}$ and a subsequence $\{\bar s_j\}_{j=0}^\infty$ such that, as $j \to \infty$, $\bar s_j \downarrow 0$, and, for all $j \in \BBN$,
		\begin{equation}
			M_{c_\dagger}\bigl(x(\bar t + \bar s_j)\bigr) < 0. \label{eq:proof_cnf_witness}
		\end{equation}
		Continuity of $M_{c_\dagger} \circ x$ at $\bar t$ gives $M_{c_\dagger}(x(\bar t)) \leq 0$. Since, in addition, $M_{c_\dagger}(x(\bar t)) \geq h(x(\bar t)) = 0$ by \eqref{eq:cnf_h}, it follows that
		\begin{equation}
			M_{c_\dagger}(x(\bar t)) = 0, \quad c_\dagger \in \SA^{\rmc}_0(x(\bar t)) \subseteq \SA^{\rmc}_\delta(x(\bar t)). \nn
		\end{equation}
		Continuity of $(M_{c_\dagger} \circ x) - (h \circ x)$ at $\bar t$, combined with the strict inequality $0 < \delta$, yields that there exists $\eta > 0$ such that, for all $t \in [\bar t, \bar t + \eta]$, $c_\dagger \in \SA^{\rmc}_\delta(x(t))$.
		
		Define $\mu \colon [\bar t, \bar t + \eta] \to \BBR$ by $\mu(t) \triangleq M_{c_\dagger}(x(t)) = \max_{i \in \mathcal{S}_{c_\dagger}} h_i(x(t))$. Since $u_*$ is locally Lipschitz and $f$, $G$ are continuous, $x(\cdot)$ is continuously differentiable, and each map $t \mapsto h_i(x(t))$ is continuously differentiable; hence $\mu$, as their pointwise maximum, is locally Lipschitz, absolutely continuous, and differentiable at almost every $t \in [\bar t, \bar t + \eta]$. Since $\mu$ is continuous and $\mu(\bar t) = 0$, shrink $\eta$ if necessary so that $\mu$ takes values in the domain of $\alpha$. Let $t \in (\bar t, \bar t + \eta)$ be a point of differentiability of $\mu$, and let $i \triangleq i_{c_\dagger}(x(t))$, so that, by Definition~\ref{def:clause_active}, $h_i(x(t)) = \mu(t)$ and, for all $\tau \in [\bar t, \bar t + \eta]$, $h_i(x(\tau)) \leq \mu(\tau)$. The map $\tau \mapsto \mu(\tau) - h_i(x(\tau))$ is nonnegative, attains the value zero at $\tau = t$, and is differentiable at $t$; its derivative at the interior minimizer $t$ therefore vanishes, and
		\begin{equation}
			\dot\mu(t) = \Lf{f}{h_i}(x(t)) + \Lf{G}{h_i}(x(t))\, u_*(x(t)). \nn
		\end{equation}
		Since $c_\dagger \in \SA^{\rmc}_\delta(x(t))$, \eqref{eq:cnf_CBF} implies that, for almost every $t \in [\bar t, \bar t + \eta]$,
		\begin{equation}
			\dot\mu(t) \geq -\alpha\bigl(h_i(x(t))\bigr) = -\alpha(\mu(t)). \label{eq:proof_cnf_diffineq}
		\end{equation}
		Now let $j \in \BBN$ be such that $\bar s_j \in (0,\eta)$, and let $t_1 \triangleq \bar t + \bar s_j$, so that \eqref{eq:proof_cnf_witness} gives $\mu(t_1) < 0$. Define $t_2 \triangleq \sup\{s \in [\bar t, t_1] \colon \mu(s) \geq 0\}$. Continuity of $\mu$ and $\mu(\bar t) = 0$ yield $t_2 \in [\bar t, t_1)$, $\mu(t_2) = 0$, and, for all $s \in (t_2, t_1]$, $\mu(s) < 0$. Since $\alpha \in \SK_e$ is increasing with $\alpha(0) = 0$, it follows that, for almost every $s \in (t_2, t_1)$, $\dot\mu(s) \geq -\alpha(\mu(s)) \geq 0$. Absolute continuity of $\mu$ then yields
		\begin{equation}
			\mu(t_1) = \mu(t_2) + \int_{t_2}^{t_1} \dot\mu(s)\, \rmd s \geq 0, \nn
		\end{equation}
		which contradicts $\mu(t_1) < 0$. Hence no escape time $\bar t$ exists, and $\SC$ is forward invariant with respect to the closed-loop system.
	\end{proof}
	
	\begin{remark}\rm
		\label{rem:cnf_enforcement}
		Theorem~\ref{thm:cnf_invariance} completes the Boolean extension: Theorem~\ref{thm:composition} locates the extracted gradient in the active-gradient polytope, and Theorem~\ref{thm:cnf_invariance} certifies forward invariance. The enforcement rule follows the tree structure: every $\delta$-active clause is enforced, and, within each clause, only the routed leaf is enforced, since forward invariance of a union follows from forward invariance of one member. The QP therefore carries at most $|\SA^{\rmc}_\delta(x)| \leq q$ rows, and the clause values $M_c(x)$ and the routed Lie derivatives are intermediates of the $m+1$ seeded passes of Section~\ref{sec:QP}, so Proposition~\ref{prop:complexity} applies with $C_h$ the cost of evaluating \eqref{eq:cnf_h}. Switching of the routed leaf within a clause requires no hysteresis: the proof of Theorem~\ref{thm:cnf_invariance} accounts for it through the derivative of the clause value at its points of differentiability. The conjunctive representation of a given $\Phi \in \mathcal{F}_p$ may enlarge the expression; for specifications posed in conjunctive form, the bound applies to the given representation. Hybrid formulations of Boolean compositions are developed in \cite{glotfelter2019hybrid}; Theorem~\ref{thm:cnf_invariance} obtains the guarantee for conjunctive-form specifications within a single continuous formulation. \exampletriangle
	\end{remark}
	
	\subsection{Higher Relative Degree via Truncated Polynomial Duals}
	
	If the system has relative degree $r > 1$, then the control input does not appear in the first $r-1$ time derivatives of the safety margin; the safety condition instead involves the iterated Lie derivatives $L_f^j h$ for $j \in \{1,\ldots,r\}$. We extend the dual-number ring so that this entire sequence is produced by a single evaluation of $h$.
	
	\begin{definition}
		\label{def:hyper_dual}
		For any relative degree $r \in \BBZ_+$, the \textit{truncated polynomial dual ring} of order $r$ is the quotient structure
		\begin{equation}
			\BBD_{(r)} \triangleq \BBR[\epsilon] / (\epsilon^{r+1}). \nn
		\end{equation}
		Every element $z \in \BBD_{(r)}$ is uniquely represented by the polynomial $z = \sum_{j=0}^r a_j \epsilon^j$ with $a_j \in \BBR$, subject to the strict nilpotent condition $\epsilon^{r+1} = 0$, while $\epsilon^j \neq 0$ for all $j \leq r$.
	\end{definition}
	
	The structure $\BBD_{(r)}$ is a local ring whose unique maximal ideal is generated by $\epsilon$, and arithmetic follows polynomial multiplication modulo $\epsilon^{r+1}$. The dual-number ring of Definition~\ref{def:dual_numbers} is the case $\BBD_{(1)} = \BBD$. The \textit{dual extension} of a $C^r$ function $\phi \colon \BBR \to \BBR$ to $\BBD_{(r)}$ is defined for $z = \sum_{k=0}^r a_k \epsilon^k$ by the truncated Taylor expansion
	\begin{equation}
		\tilde\phi(z) \triangleq \sum_{j=0}^{r} \frac{\phi^{(j)}(a_0)}{j!}\, (z - a_0)^{j} \pmod{\epsilon^{r+1}}, \label{eq:phi_truncated}
	\end{equation}
	where $(z - a_0)^j$ is evaluated within $\BBD_{(r)}$. The composition rule of Lemma~\ref{lem:composition} extends to $\BBD_{(r)}$ for all $r \geq 1$ by induction on the operational syntax tree.
	
	The following theorem shows that seeding this higher-order arithmetic with the Taylor jet of the flow returns the complete sequence of iterated Lie derivatives.
	
	\begin{theorem}
		\label{thm:higher_rel_deg}
		Let $h \colon \SD \to \BBR$ be $r$-times continuously differentiable, and let the system vector field $f \colon \SD \to \BBR^n$ be $(r-1)$-times continuously differentiable. Define the kinematic state iterates as
		\begin{equation}
			f^{[0]}(x) \triangleq x, \quad f^{[j]}(x) \triangleq \nabla f^{[j-1]}(x)\, f(x), \label{eq:flow_jets}
		\end{equation}
		where $j \in \{1,\ldots,r\}$. Define the \textit{order-$r$ dual flow seed} vector
		\begin{equation}
			x_{\BBD_{(r)},f} \triangleq \sum_{j=0}^r \frac{1}{j!}\, f^{[j]}(x)\, \epsilon^j \in \BBD_{(r)}^n. \label{eq:hyper_seed}
		\end{equation}
		Then, the iterated Lie series is given by
		\begin{equation}
			\tilde h\bigl(x_{\BBD_{(r)},f}\bigr) = \sum_{j=0}^r \frac{1}{j!}\, L_f^j h(x)\, \epsilon^j. \label{eq:hyper_result}
		\end{equation}
		Specifically, for all $j\in\{0,\ldots,r\}$, $L_f^j h(x) = c_jj! $, where $c_j$ is the scalar coefficient of $\epsilon^j$ in the exact evaluated polynomial $\tilde h\bigl(x_{\BBD_{(r)},f}\bigr)$.
	\end{theorem}
	
	\begin{proof}
		For all $t \ge 0$, let $\varphi_t \colon \SD \to \SD$ denote the flow satisfying $\frac{\rmd}{\rmd t}\varphi_t(x) = f(\varphi_t(x))$ with $\varphi_0(x) = x$. Since $f$ is $(r-1)$-times continuously differentiable, the map $t \mapsto \varphi_t(x)$ is $r$-times continuously differentiable, and differentiating the flow identifies its time derivatives with the spatial iterates of \eqref{eq:flow_jets}:
		\begin{equation}
			\left.\frac{\rmd^j}{\rmd t^j}\varphi_t(x)\right|_{t=0} = f^{[j]}(x), \quad j \in \{0,\ldots,r\}. \label{eq:proof_jet}
		\end{equation}
		Let $P_r(t) \triangleq \sum_{j=0}^{r} \frac{1}{j!}\, f^{[j]}(x)\, t^j$ denote the order-$r$ Taylor polynomial of the trajectory at $t=0$. The dual seed \eqref{eq:hyper_seed} is this polynomial evaluated at the nilpotent variable:
		\begin{equation}
			x_{\BBD_{(r)},f} = P_r(\epsilon). \label{eq:proof_seed_poly}
		\end{equation}
		Because $h \in C^r$ and $t \mapsto \varphi_t(x) \in C^r$, the composition $h \circ \varphi_t$ is $r$-times continuously differentiable, and the chain rule gives $\frac{\rmd^j}{\rmd t^j} h(\varphi_t(x))|_{t=0} = L_f^j h(x)$ for $j \in \{0,\ldots,r\}$. Taylor's theorem then yields, as $t \to 0$,
		\begin{equation}
			h(\varphi_t(x)) = \sum_{j=0}^{r} \frac{1}{j!}\, L_f^j h(x)\, t^j + o(t^r). \label{eq:proof_lie_series}
		\end{equation}
		The dual extension \eqref{eq:phi_truncated} evaluates $\tilde h$ as the order-$r$ multivariable Taylor expansion of $h$ reduced modulo $\epsilon^{r+1}$. Since $P_r(\epsilon) - x = \sum_{j=1}^{r} \frac{1}{j!} f^{[j]}(x)\,\epsilon^j$ has no constant term, its powers above order $r$ vanish in $\BBD_{(r)}$, and
		\begin{equation}
			\tilde h\bigl(P_r(\epsilon)\bigr) = \sum_{k=0}^r \frac{1}{k!}\, D^k h(x) \bigl[ P_r(\epsilon) - x \bigr]^k \pmod{\epsilon^{r+1}}. \label{eq:proof_algebraic_comp}
		\end{equation}
		The right-hand side of \eqref{eq:proof_algebraic_comp} is the order-$r$ truncation of the Taylor expansion of $h \circ \varphi_t$ with $t$ replaced by $\epsilon$; matching it term by term with \eqref{eq:proof_lie_series}, in which the remainder $o(t^r)$ maps to a multiple of $\epsilon^{r+1} = 0$, the coefficient of $\epsilon^j$ equals $\frac{1}{j!} L_f^j h(x)$ for each $j \in \{0,\ldots,r\}$, which is \eqref{eq:hyper_result}.
	\end{proof}
	
	\begin{remark}\rm
		\label{rem:ECBF_extension}
		For a system of relative degree $r$, the high-order CBF constructions \cite{xiao2019high,xiao2022high} cascade the envelopes
		\begin{equation}
			\psi_0(x) \triangleq h(x), \quad \psi_k(x) \triangleq \dot\psi_{k-1}(x) + \alpha_k(\psi_{k-1}(x)), \nn
		\end{equation}
		for $k \in \{1,\ldots,r\}$ with $\alpha_k \in \SK_e$; the exponential CBF (ECBF) construction \cite{nguyen2016exponential} corresponds to linear functions $\alpha_k$ obtained by pole placement. The actionable condition $\psi_r(x,u) \geq 0$ is affine in $u$ through the single coupling term $\Lf{G}{L_f^{r-1} h}(x)$. Theorem~\ref{thm:higher_rel_deg} supplies the drift terms $L_f h(x), \ldots, L_f^{r} h(x)$ from the single seed $x_{\BBD_{(r)},f}$. The coupling term $\Lf{G}{L_f^{r-1} h}(x)$ is not available from the seed \eqref{eq:hyper_seed}: Theorem~\ref{thm:higher_rel_deg} returns the values $L_f^j h(x)$, whereas lifting $L_f^{r-1} h$ along an input column requires the derivative of the map $x \mapsto L_f^{r-1} h(x)$, which the univariate arithmetic does not represent. Proposition~\ref{prop:coupling} below extracts the coupling by a bivariate truncated-dual evaluation that seeds the drift jet and one input column simultaneously; in the relative-degree-one case it reduces to the control seeds of Remark~\ref{rem:control_seed}. When $h$ is a pointwise minimum or a conjunctive-form composition \eqref{eq:cnf_h}, the routing of Theorems~\ref{thm:vertex_selection} and \ref{thm:composition} is applied at the outer reduction, and the simultaneous enforcement of Theorems~\ref{thm:forward_invariance} and \ref{thm:cnf_invariance} is imposed on each $\delta$-active constraint or clause. \exampletriangle
	\end{remark}

	\begin{definition}
		\label{def:bivariate_ring}
		For $r \in \BBZ_+$, the \textit{bivariate truncated dual ring} is the quotient
		\begin{equation}
			\BBD_{(r-1,1)} \triangleq \BBR[\epsilon_1,\epsilon_2]/(\epsilon_1^{r},\, \epsilon_2^{2}). \nn
		\end{equation}
		Every element is uniquely represented as $z = \sum_{k=0}^{r-1} (a_k + b_k \epsilon_2)\, \epsilon_1^k$ with $a_k, b_k \in \BBR$. Equivalently, $\BBD_{(r-1,1)}$ is the truncated polynomial ring of order $r-1$ in $\epsilon_1$ whose coefficients lie in the dual ring $\BBD = \BBR[\epsilon_2]/(\epsilon_2^2)$; for $\phi \in C^r$, the dual extension of $\phi$ to $\BBD_{(r-1,1)}$ is \eqref{eq:phi_truncated} of order $r-1$, with each coefficient $\phi^{(j)}(a_0)$ evaluated through its $\epsilon_2$-dual extension \eqref{eq:dual_ext_scalar} at the $\epsilon_1$-constant term $a_0 + b_0 \epsilon_2$.
	\end{definition}
	
	\begin{proposition}
		\label{prop:coupling}
		Let $h \colon \SD \to \BBR$ be $r$-times continuously differentiable, let $f \colon \SD \to \BBR^n$ be $(r-1)$-times continuously differentiable, let $G \colon \SD \to \BBR^{n \times m}$ be continuous, and, for $j \in \{1,\ldots,m\}$, let $g \triangleq \col_j(G(x))$. Define the \textit{bivariate seed}
		\begin{equation}
			z_g \triangleq \sum_{k=0}^{r-1} \frac{1}{k!}\, \tilde f^{[k]}\bigl(x + g\, \epsilon_2\bigr)\, \epsilon_1^k \in \BBD_{(r-1,1)}^n, \label{eq:bivariate_seed}
		\end{equation}
		where $\tilde f^{[k]}$ is the $\epsilon_2$-dual extension \eqref{eq:dual_ext_vec} of the iterate $f^{[k]}$ in \eqref{eq:flow_jets}. Then,
		\begin{equation}
			\tilde h(z_g) = \sum_{k=0}^{r-1} \frac{1}{k!} \Bigl( L_f^k h(x) + \Lf{g}{L_f^k h}(x)\, \epsilon_2 \Bigr)\, \epsilon_1^k. \label{eq:bivariate_result}
		\end{equation}
		In particular, $\Lf{\col_j(G)}{L_f^{r-1} h}(x) = (r-1)!\, b_{r-1}$, where $b_{r-1}$ is the coefficient of $\epsilon_1^{r-1} \epsilon_2$ in $\tilde h(z_g)$.
	\end{proposition}
	
	\begin{proof}
		For $s \in \BBR$ in a neighborhood of $0$, define
		\begin{equation}
			\zeta(s) \triangleq \sum_{k=0}^{r-1} \frac{1}{k!}\, f^{[k]}(x + s g)\, \epsilon_1^k \in \BBD_{(r-1)}^n. \nn
		\end{equation}
		Theorem~\ref{thm:higher_rel_deg}, applied at the base point $x + s g$ with order $r-1$, yields
		\begin{equation}
			\tilde h\bigl(\zeta(s)\bigr) = \sum_{k=0}^{r-1} \frac{1}{k!}\, L_f^k h(x + s g)\, \epsilon_1^k. \label{eq:proof_biv_family}
		\end{equation}
		Since $h$ is $C^r$ and $f$ is $C^{r-1}$, the iterates satisfy $f^{[k]} \in C^{r-k}$ and $L_f^k h \in C^{r-k}$; hence, for all $k \leq r-1$, the maps $s \mapsto f^{[k]}(x + s g)$ and $s \mapsto L_f^k h(x + s g)$ are continuously differentiable.
		
		Let $\Psi \colon \BBR^{nr} \to \BBR^{r}$ denote the map that sends the $\epsilon_1$-coefficients of a seed in $\BBD_{(r-1)}^n$ to the $\epsilon_1$-coefficients of $\tilde h$ evaluated at that seed. Each component of $\Psi$ is a finite composition of the coefficient operations of the truncated polynomial arithmetic of Definition~\ref{def:hyper_dual} and of the derivatives $\phi^{(k)}$, $k \leq r-1$, of the $C^r$ primitives of $h$; each component of $\Psi$ is therefore continuously differentiable. By Definition~\ref{def:bivariate_ring}, evaluating $\tilde h$ over $\BBD_{(r-1,1)}$ performs every coefficient operation of the $\BBD_{(r-1)}$ evaluation in $\epsilon_2$-dual arithmetic; by Lemma~\ref{lem:composition} and Remark~\ref{rem:multivar}, the resulting composition equals the $\epsilon_2$-dual extension $\tilde\Psi$ of $\Psi$. By \eqref{eq:dual_ext_vec}, the $\epsilon_2$-parts of the output of $\tilde\Psi$ equal the derivative of $\Psi$ contracted with the $\epsilon_2$-parts of its argument.
		
		The seed \eqref{eq:bivariate_seed} has $\epsilon_1$-coefficients
		\begin{equation}
			\tfrac{1}{k!}\, \tilde f^{[k]}(x + g \epsilon_2) = \tfrac{1}{k!} \Bigl( f^{[k]}(x) + \tfrac{\rmd}{\rmd s} f^{[k]}(x + s g)\big|_{s=0}\, \epsilon_2 \Bigr), \nn
		\end{equation}
		that is, its $\epsilon_2$-parts are the derivatives at $s = 0$ of the $\epsilon_1$-coefficients of $\zeta(s)$. Applying the chain rule to the composition of $\Psi$ with $s \mapsto \zeta(s)$ and using \eqref{eq:proof_biv_family},
		\begin{equation}
			\tilde h(z_g) = \sum_{k=0}^{r-1} \frac{1}{k!} \Bigl( L_f^k h(x) + \tfrac{\rmd}{\rmd s} L_f^k h(x + s g)\big|_{s=0}\, \epsilon_2 \Bigr)\, \epsilon_1^k. \nn
		\end{equation}
		Finally, $\tfrac{\rmd}{\rmd s} L_f^k h(x + s g)\big|_{s=0} = \nabla\bigl(L_f^k h\bigr)(x)\, g = \Lf{g}{L_f^k h}(x)$, which is \eqref{eq:bivariate_result}; isolating $k = r-1$ yields the expression for the coupling.
	\end{proof}
	
	\begin{remark}\rm
		\label{rem:jet_recursion}
		The seed \eqref{eq:bivariate_seed} requires no symbolic differentiation. The iterates $\tilde f^{[k]}(x + g \epsilon_2)$ need not be formed from programs for $\nabla f^{[k-1]}$: the Taylor coefficients of the flow of $\dot x = f(x)$ obey the recursion
		\begin{equation}
			c_0 = x + g\, \epsilon_2, \qquad c_{k+1} = \frac{1}{k+1} \Bigl[\, \tilde f\Bigl(\textstyle\sum_{i=0}^{k} c_i\, \epsilon_1^i\Bigr) \Bigr]_{\epsilon_1^k}, \nn
		\end{equation}
		where $[\,\cdot\,]_{\epsilon_1^k}$ extracts the coefficient of $\epsilon_1^k$ and the arithmetic is carried out in $\BBD_{(r-1,1)}$ on the straight-line program of $f$ \cite[Chap.\ 13]{griewank2008evaluating}; then $c_k = \frac{1}{k!} \tilde f^{[k]}(x + g \epsilon_2)$. The same recursion with $c_0 = x$ produces the seed \eqref{eq:hyper_seed}. Every element of $\BBD_{(r-1,1)}$ carries $2r$ real coefficients, so each ring operation costs a number of real operations that depends on $r$ only; the coupling extraction therefore preserves the $O(C_h)$ bound of Proposition~\ref{prop:complexity}, with a constant that depends on $r$ and remains independent of $n$, $p$, and $|\SA(x)|$. Extraction of the row $\Lf{G}{L_f^{r-1} h}(x)$ requires one bivariate evaluation per input column, which mirrors the $m$ control seeds of Remark~\ref{rem:control_seed}. When $h$ is a pointwise minimum, the $\minRe$ reduction is applied to the bivariate evaluations as in Definition~\ref{def:dual_eval}, comparing the coefficients of $\epsilon_1^0 \epsilon_2^0$ only. \exampletriangle
	\end{remark}

	\section{QP-Based Safety Filter and Complexity}
	\label{sec:QP}
	
	This section presents the QP-based safety filter that implements the dual-algebraic forward-invariance result, and quantifies its computational cost. Figure~\ref{fig:pipeline} summarizes the data flow from the measured state to the safety-filter output.
	
	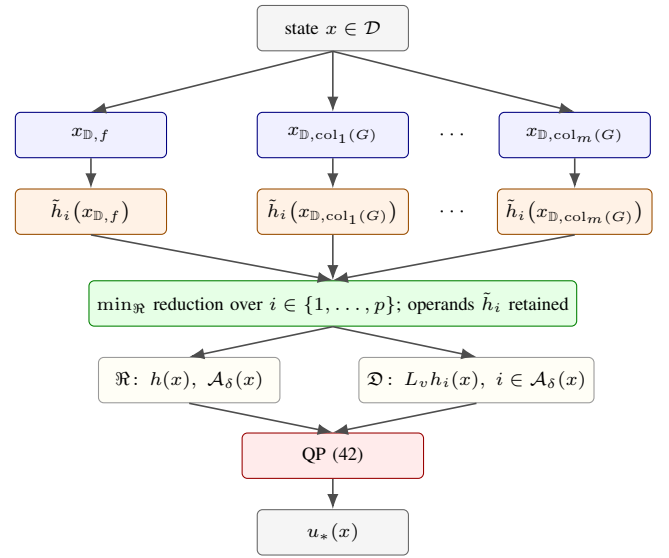
\begin{figure}[ht!]
		\centering
		\begin{tikzpicture}[
			>=Latex,
			node distance = 4mm and 3mm,
			every node/.style={font=\scriptsize},
			block/.style={draw, rounded corners=2pt, align=center,
				inner sep=3pt, minimum height=6mm},
			io/.style={block, fill=gray!8,    draw=black!60,        minimum width=20mm},
			seed/.style={block, fill=blue!6,  draw=blue!50!black,   minimum width=20mm},
			eval/.style={block, fill=orange!10, draw=orange!60!black, minimum width=20mm},
			reduce/.style={block, fill=green!10, draw=green!50!black,  minimum width=56mm},
			output/.style={block, fill=yellow!4, draw=black!40, minimum width=24mm, font=\scriptsize},
			qp/.style={block, fill=red!8,    draw=red!60!black,    minimum width=24mm},
			arr/.style={->, semithick, draw=black!70},
			]
			\node[io] (state) {state $x \in \SD$};
			\node[seed, below=8mm of state, xshift=-32mm]                  (sf)  {$x_{\BBD,f}$};
			\node[seed, below=8mm of state]                                (sg1) {$x_{\BBD,\col_1(G)}$};
			\node[seed, below=8mm of state, xshift= 32mm]                   (sgp) {$x_{\BBD,\col_m(G)}$};
			\node[eval, below=4mm of sf]   (ef)  {$\tilde h_i\big(x_{\BBD,f}\big)$};
			\node[eval, below=4mm of sg1]  (eg1) {$\tilde h_i\big(x_{\BBD,\col_1(G)}\big)$};
			\node[eval, below=4mm of sgp]  (egp) {$\tilde h_i\big(x_{\BBD,\col_m(G)}\big)$};
			\node at ($(sg1)!0.5!(sgp)$)   {$\cdots$};
			\node at ($(eg1)!0.5!(egp)$)   {$\cdots$};
			\node[reduce, below=6mm of eg1] (red) {$\minRe$ reduction over $i \in \{1,\ldots,p\}$; operands $\tilde h_i$ retained};
			\node[output, below=4mm of red, xshift=-19mm] (out1) {$\Re\colon\, h(x),\ \SA_\delta(x)$};
			\node[output, below=4mm of red, xshift= 19mm] (out2) {$\Du\colon\, \Lf{v}{h_i}(x),\ i \in \SA_\delta(x)$};
			\node[qp,  below=4mm of red, yshift=-10mm] (qp) {QP \eqref{eq:QP_simult}};
			\node[io,  below=4mm of qp] (u)  {$u_*(x)$};
			\draw[arr] (state.south) -- (sf.north);
			\draw[arr] (state.south) -- (sg1.north);
			\draw[arr] (state.south) -- (sgp.north);
			\draw[arr] (sf)  -- (ef);
			\draw[arr] (sg1) -- (eg1);
			\draw[arr] (sgp) -- (egp);
			\draw[arr] (ef.south)  -- (red.north);
			\draw[arr] (eg1.south) -- (red.north);
			\draw[arr] (egp.south) -- (red.north);
			\draw[arr] (red.south) -- (out1.north);
			\draw[arr] (red.south) -- (out2.north);
			\draw[arr] (out1.south) -- (qp.north);
			\draw[arr] (out2.south) -- (qp.north);
			\draw[arr] (qp) -- (u);
		\end{tikzpicture}
		\caption{Dual-algebraic safety filter pipeline. The state $x$ is lifted along $f$ and along each column $\col_j(G)$ of the input matrix (blue); the dual-extended barriers $\tilde h_i$ are evaluated on every seed for every $i \in \{1,\ldots,p\}$ (orange); the $\minRe$ reduction (green) returns $h(x)$ and retains its operands $\tilde h_i$; the real parts determine $\SA_\delta(x)$, and the dual parts supply the Lie derivatives of every $i \in \SA_\delta(x)$; these data feed the safety-filter QP \eqref{eq:QP_simult} (red), whose output is $u_*(x)$. By Proposition~\ref{prop:complexity}, the cost is $O(C_h)$ per seed, independent of $|\SA(x)|$ and of the local geometry of $\bd\SC$.}
		\label{fig:pipeline}
	\end{figure}
	
	\subsection{Safety-Filter QP}
	
	Given a nominal controller $u_{\rm n} \colon \SD \to \BBR^m$ and a margin $\delta > 0$, the \textit{dual-algebraic safety filter} solves
	\begin{align}
		u_*(x) = \argmin_{u \in \BBR^m} \quad & \|u - u_{\rm n}(x)\|^2 \label{eq_u_*}
	\end{align}
	such that, for all $i \in \SA_\delta(x)$,
	\begin{align}
		\Du\bigl(\tilde h_i(x_{\BBD,f})\bigr) + \sum_{j=1}^{m} \Du\bigl(\tilde h_i(x_{\BBD,\col_j(G)})\bigr)\, u_{(j)}  + \alpha(h_i(x)) \geq 0. \label{eq:QP_simult}
	\end{align}
	By Lemma~\ref{lem:dual_eval_structure}, the dual parts in \eqref{eq:QP_simult} equal $\Lf{f}{h_i}(x)$ and $\Lf{\col_j(G)}{h_i}(x)$, respectively. By Theorem~\ref{thm:forward_invariance}, the safe set $\SC$ is forward invariant with respect to the closed-loop system with all locally Lipschitz control $u_*$ satisfying \eqref{eq_u_*} subject to \eqref{eq:QP_simult}.
	
	Populating \eqref{eq:QP_simult} requires no evaluations beyond the $m+1$ seeded passes: the composite evaluation \eqref{eq:dual_eval} computes every operand $\tilde h_i(x_{\BBD,v})$ before the $\minRe$ reduction, so the $p$ pairs $\bigl(h_i(x),\, \Lf{v}{h_i}(x)\bigr)$ are intermediates of each pass and are retained in a buffer of $2p$ scalars per seed. The reduction determines which pair propagates to the composite output; it does not erase the remaining pairs. Reading $\SA_\delta(x)$ from the real parts and the coefficients of \eqref{eq:QP_simult} from the dual parts adds $O(p) \subseteq O(C_h)$ copies and comparisons, which preserves the operation bound of Proposition~\ref{prop:complexity} and the branch-free, allocation-free structure of the procedure.
	
	\begin{remark}\rm
		\label{rem:single_constraint}
		When the trajectory is bounded away from the non-smooth locus, the active set reduces to a singleton, $\SA_\delta(x) = \{i_*(x)\}$ for sufficiently small $\delta > 0$, and \eqref{eq:QP_simult} reduces to a single-constraint QP. The unique constraint is the one identified by the lexicographic selection of Lemma~\ref{lem:dual_eval_structure}, and the QP admits the closed-form solution
		\begin{equation}
			u_* = u_{\rm n} - \min\!\left\{0,\ \frac{a + b\, u_{\rm n}}{\|b\|^2}\right\} b^\rmT, \nn
		\end{equation}
		with
		\begin{align}
			a &\triangleq \Du\bigl(\tilde h_{i_*}(x_{\BBD,f})\bigr) + \alpha(h(x)), \nn\\
			b &\triangleq \Big[\Du\Bigl(\tilde h_{i_*}(x_{\BBD,\col_1(G)})\Bigr)~ \cdots~ \Du\Bigl(\tilde h_{i_*}\big(x_{\BBD,\col_m(G)}\big)\Bigr)\Big]. \nn
		\end{align}
		\exampletriangle
	\end{remark}

	\subsection{Computational Complexity}
	
	\begin{proposition}
		\label{prop:complexity}
		Let $C_h$ be the number of elementary arithmetic operations (i.e., additions, multiplications, and comparisons) required to evaluate the composite barrier $h(x)$ over $\BBR$. For any continuous vector field $v \colon \SD \to \BBR^n$, the dual extension $\tilde h(x_{\BBD,v})$ is evaluated in at most $4 C_h$ real operations, an overhead of at most $3 C_h$, independent of the state dimension $n$, the constraint count $p$, the active-set cardinality $|\SA(x)|$, and the direction $v$.
	\end{proposition}
	
	\begin{proof}
		The dual evaluation traverses the straight-line program of $h$ once, replacing each real operation by its dual counterpart, and we tally the cost of each. Dual addition $(a_1 + b_1\epsilon) + (a_2 + b_2\epsilon) = (a_1+a_2) + (b_1+b_2)\epsilon$ uses two real additions, a factor of two over the scalar baseline. Dual multiplication $(a_1 + b_1\epsilon)(a_2 + b_2\epsilon) = a_1 a_2 + (a_1 b_2 + a_2 b_1)\epsilon$ uses three real multiplications and one real addition, a factor of four. The reduction $\minRe(z_1, z_2)$ uses one real comparison and one conditional copy of a two-field record, a factor of two. The largest per-operation factor is four, so the dual evaluation uses at most $4 C_h$ real operations. None of these counts depends on $n$, $p$, $|\SA(x)|$, or $v$.
	\end{proof}
	
	\begin{remark}\rm
		\label{rem:vs_alternatives}
		By contrast, constructing the Clarke subdifferential $\partial h(x)$ explicitly requires $O(p)$ gradient evaluations followed by a convex-hull computation with worst-case cost $O(|\SA(x)|^{\lfloor n/2 \rfloor})$ \cite{barber1996quickhull}; the mixed-integer formulation of \cite{chen2018obstacle} introduces one binary variable per constraint, with worst-case solve time $O(2^p)$; the formulation of \cite{glotfelter2017nonsmooth} enforces one affine constraint per component function at every state, so the QP row count scales with $p$; and smooth surrogates require the selection of a parameter $\tau$, which trades approximation error against conditioning. The dual-algebraic evaluation avoids these costs while preserving exactness and determinism. \exampletriangle
	\end{remark}

	\section{Examples}
	\label{sec:examples}
	
	This section illustrates the framework on three examples of increasing scope. The mobile robot of Example~\ref{ex:rectangle} exercises vertex selection of Theorem~\ref{thm:vertex_selection} and forward invariance of Theorem~\ref{thm:forward_invariance} on a polytope, and resolves at the same geometry the failure exhibited by the counterexample of Example~\ref{ex:counter}. The multi-agent problem of Example~\ref{ex:multi_agent} scales the construction to $p = \binom{N}{2}$ constraints and the joint QP to $\BBR^{2N}$, exposing the $O(p)$ versus combinatorial cost of Proposition~\ref{prop:complexity}. The double integrator of Example~\ref{ex:doubleint} operates at relative degree two, combining the truncated-dual extraction of Theorem~\ref{thm:higher_rel_deg} with vertex selection and simultaneous enforcement at a non-smooth point. In every example the safety filter solves \eqref{eq_u_*}--\eqref{eq:QP_simult} with the Lie derivatives obtained from the dual-algebraic evaluation, and all numerical data needed to reproduce the figures are stated in full.
	
	\begin{example}
		\label{ex:rectangle}
		\rm
		Consider a mobile robot with single-integrator dynamics on $\BBR^2$,
		\begin{equation}
			\dot x = u. \label{eq:ex1_dynamics}
		\end{equation}
		The safe set is the closed rectangle $[0, \ell_1] \times [0, \ell_2]$, encoded by
		\begin{equation}
			h(x) = \min_{i\in\{1,\ldots,4\}}h_i(x), \nn
		\end{equation}
		where $h_1(x) = x_{(1)}$, $h_2(x) = \ell_1 - x_{(1)}$, $h_3(x) = x_{(2)}$, $h_4(x) = \ell_2 - x_{(2)}$, with constant gradients $\nabla h_1 = [1~\, 0]$, $\nabla h_2 = [-1~\, 0]$, $\nabla h_3 = [0~\, 1]$, $\nabla h_4 = [0~\, -1]$. The system has $f \equiv 0$ and $G \equiv I_2$; thus $\col_1(G) = e_1$ and $\col_2(G) = e_2$.
		
		At the corner $x = (0,0)$, the active set is $\SA(x) = \{1,3\}$ and the Clarke subdifferential is $\partial h(x) = \Co\{[1~0],\, [0~1]\}$. The dual seeds for the control columns are $x_{\BBD,e_1} = x + e_1 \epsilon$ and $x_{\BBD,e_2} = x + e_2 \epsilon$. Setting $\ell_1 = \ell_2 = 1$ and evaluating the dual extension of each $h_i$ at $x_{\BBD,e_1}$ yields
		\begin{align}
			\tilde h_1(x_{\BBD,e_1}) &= 0 + 1\cdot\epsilon, &\tilde h_2(x_{\BBD,e_1}) &= 1 - 1\cdot\epsilon, \nn \\
			\tilde h_3(x_{\BBD,e_1}) &= 0 + 0\cdot\epsilon, &\tilde h_4(x_{\BBD,e_1}) &= 1 + 0\cdot\epsilon. \nn
		\end{align}
		The real parts $(0, 1, 0, 1)$ have minimum $0$, attained by indices $1$ and $3$. The left-to-right reduction $\minRe$ retains index $1$, yielding $\tilde h(x_{\BBD,e_1}) = 0 + 1 \cdot \epsilon$ and hence $\Lf{e_1}{h_{i_*}}(x) = 1$, that is, $\nabla h_1(x) \cdot e_1 = 1$. By Theorem~\ref{thm:vertex_selection}, the gradient $\nabla h_1(x) = [1~0]$ is a vertex of $\partial h(0)$, as expected. As Example~\ref{ex:counter} cautions, the single-vertex enforcement here would leave $u_{(2)}$ unconstrained and admit $u = (0,-1)$, which immediately violates $h_3$. The simultaneous QP \eqref{eq:QP_simult} with $\delta = 0$ enforces both
		\begin{equation}
			u_{(1)} + \alpha(0) \geq 0, \qquad u_{(2)} + \alpha(0) \geq 0, \nn
		\end{equation}
		and thereby renders $\SC$ forward invariant for all locally Lipschitz nominal controllers, as stated in Theorem~\ref{thm:forward_invariance}.
		
		At an interior point such as $x = (0.1, 0.5)$, the active set is $\SA(x) = \{1\}$, the dual evaluation uniquely returns $\tilde h(x_{\BBD,e_1}) = 0.1 + 1 \cdot \epsilon$, and the smooth single-constraint CBF behavior is recovered.
		
		\textit{Numerical simulation.} We integrate \eqref{eq:ex1_dynamics} on $\SC = [0~1]^2$ under the dual-algebraic safety filter \eqref{eq:QP_simult} from $x(0) = (0.5, 0.5)$ over $T = 20\,\mathrm{s}$ with explicit Euler at step $\Delta t = 10^{-3}\,\mathrm{s}$. The class-$\SK_e$ function is $\alpha(h) = 5 h$, the $\delta$-active margin is $\delta = 5 \times 10^{-2}$, and the nominal controller is the proportional law $u_{\rm n}(x,t) = 0.5\bigl(x_{\rm ref}(t) - x\bigr)$, where the reference $x_{\rm ref}(t)$ steps every $5\,\mathrm{s}$ through the waypoint sequence $(-0.5,-0.5)$, $(1.5,-0.5)$, $(1.5,1.5)$, $(-0.5,1.5)$. Each waypoint lies outside $\SC$, so the safety filter must engage along every edge and at every corner.
		
		Figure~\ref{fig:rect_phase} shows the closed-loop trajectory: the robot crosses the interior toward the bottom-left corner, then slides counterclockwise along the perimeter, transiently activating two constraints at every corner. The minimum of $h(x(t))$ over the run is $-1.1 \times 10^{-17}$, at the level of double-precision floating-point error, consistent with the forward invariance of Theorem~\ref{thm:forward_invariance}. Figure~\ref{fig:rect_h_active} plots the four constraint values $h_i(x(t))$ and the $\delta$-active cardinality $|\SA_\delta(x(t))|$. The cardinality equals two at every corner sojourn, exactly the non-smooth events at which single-vertex enforcement fails (See Example~\ref{ex:counter}), and equals one on the smooth edges of $\bd\SC$. \exampletriangle
	\end{example}
	
	\begin{figure}[t!]
		\centering
		\includegraphics[width=0.7\columnwidth]{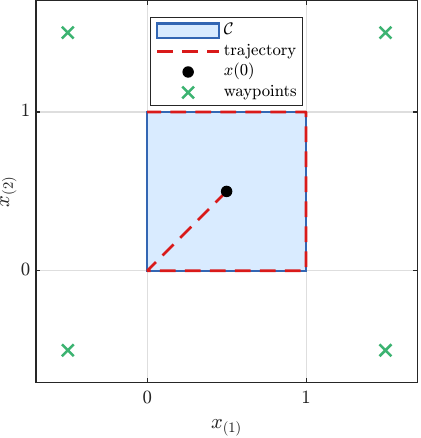}
		\caption{Closed-loop trajectory of Example~\ref{ex:rectangle}. The nominal controller tracks four waypoints $\times$ placed outside $\SC = [0,1]^2$; the dual-algebraic safety filter \eqref{eq:QP_simult} renders $\SC$ forward invariant. The trajectory slides along each edge of $\SC$ (i.e., smooth active set) and crosses each corner (i.e., non-smooth active set, two simultaneously tight constraints).}
		\label{fig:rect_phase}
	\end{figure}
	
	\begin{figure}[ht!]
		\centering
		\includegraphics[width=0.95\columnwidth]{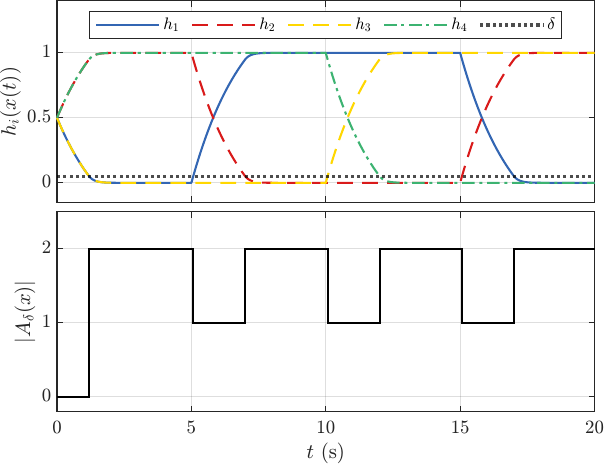}
		\caption{Constraint values and active-set cardinality for Example~\ref{ex:rectangle}. Top: the four barriers $h_i(x(t))$, each nonnegative throughout, with the dotted line at the margin $\delta$. Bottom: the $\delta$-active cardinality $|\SA_\delta(x(t))|$, which equals two at the four corner crossings, namely, $t \approx 0, 5, 10, 15\,\mathrm{s}$, and one along the edges.}
		\label{fig:rect_h_active}
	\end{figure}

	\begin{example}
		\label{ex:multi_agent}
		\rm
		Consider $N$ planar agents with single-integrator dynamics
		\begin{equation}
			\dot x_i = u_i, \nn
		\end{equation}
		where $x_i \in \BBR^2$ and $u_i\in\BBR^2$ are, respectively, the position and control input of agent $i \in \{1, \ldots, N\}$. Let $x = [x_1^\rmT~ \cdots~ x_N^\rmT]^\rmT \in \BBR^{2N}$. The pairwise collision-avoidance constraint between agents $i$ and $j$ is
		\begin{equation}
			h_{i,j}(x) \triangleq \|x_i - x_j\|^2 - d_{\min}^2, \nn
		\end{equation}
		where $d_{\min} > 0$. The composite barrier is
		\begin{equation}
			h(x) = \min_{1 \le i < j \le N} h_{i,j}(x), \nn
		\end{equation}
		and is non-smooth at all states where two or more inter-agent distance constraints are simultaneously tight. For example, if $\|x_1 - x_2\| = \|x_1 - x_3\| = d_{\min}$ with $x_2 \neq x_3$, then $\SA(x) = \{(1,2),\, (1,3)\}$.
		
		The stacked control matrix is $G = I_{2N}$, and thus, for all $k \in \{1,\ldots,2N\}$, $\col_k(G) = e_k$. For each pair, the Lie derivatives are
		\begin{gather}
			\Lf{f}{h_{i,j}}(x) = 0,\nn \\ \nabla_{x_i} h_{i,j}(x) = 2(x_i - x_j)^\rmT, \quad \nabla_{x_j} h_{i,j}(x) = -2(x_i - x_j)^\rmT, \nn
		\end{gather}
		and are extracted by the dual-algebraic evaluation without symbolic differentiation. With $N$ agents and $p = \binom{N}{2}$ pairwise constraints, the dual evaluation of the composite barrier requires $O(p) = O(N^2)$ dual operations and $O(p)$ comparisons. At a multiple-way collision event, this compares with the $O(p^{N})$ worst-case cost of constructing the Clarke subdifferential polytope of the composite barrier in $\BBR^{2N}$ \cite{barber1996quickhull}. Theorem~\ref{thm:forward_invariance} guarantees that the joint safe set is forward invariant under the simultaneous QP \eqref{eq:QP_simult}.
		
		\textit{Numerical simulation.} We simulate $N = 3$ agents with $d_{\min} = 0.7$ performing an antipodal position swap: the agents start at the vertices of an equilateral triangle of circumradius $2.2$ and are assigned goals at the antipodal vertices. For each agent $i\in\{1,2,3\}$, the nominal controller is $u_{i,\rm n} = x_{i,\rm goal} - x_i$, where $x_{i,\rm goal}$ is the antipodal position of $x_i$. The class-$\SK_e$ function is $\alpha(h) = 5h$, the margin is $\delta = 5 \times 10^{-2}$, and the dynamics are integrated by explicit Euler over $T = 20\,\mathrm{s}$ at step $\Delta t = 10^{-3}\,\mathrm{s}$. The joint safety filter \eqref{eq:QP_simult} solves a single QP in $\BBR^{6}$ at each step. Figure~\ref{fig:multi_sim} shows the agent trajectories and the three inter-agent distances. All agents reach their goals, and the minimum inter-agent distance over the run is $0.70 = d_{\min}$: the symmetric three-way conflict at the center is resolved by a coordinated rotation that holds every pair exactly on the safety boundary. \exampletriangle
	\end{example}
	
	\begin{figure}[t!]
		\centering
		\includegraphics[width=0.7\columnwidth]{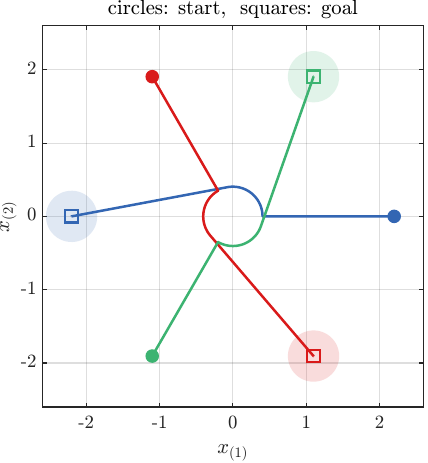}\\[0.3em]
		\includegraphics[width=0.92\columnwidth]{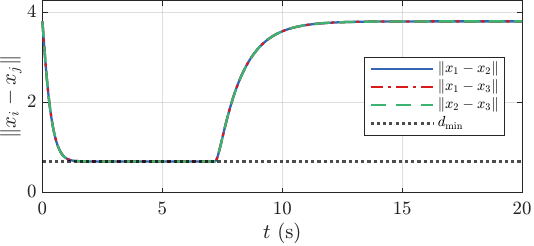}
		\caption{Closed-loop simulation of Example~\ref{ex:multi_agent} with $N = 3$ agents. Top: trajectories from start $\circ$ to goal $\square$; shaded disks have radius $d_{\min}/2$, so disjoint disks certify separation. Bottom: the three inter-agent distances remain above $d_{\min}$  at all times; the simultaneous dip near $t \approx 1\,\mathrm{s}$ is the multi-way conflict at the center, resolved by the joint dual-algebraic QP.}
		\label{fig:multi_sim}
	\end{figure}

	\begin{example}
		\label{ex:doubleint}
		\rm
		Consider a planar double integrator with state $x = [q^\rmT~\, v^\rmT]^\rmT \in \BBR^4$, where $q \in \BBR^2$ is position and $v \in \BBR^2$ is velocity, and acceleration control $u \in\BBR^2$,
		\begin{equation}
			\dot q = v, \qquad \dot v = u. \label{eq:di_dynamics}
		\end{equation}
		This has the form \eqref{eq:system} with
		\begin{equation}
			f(x) = \begin{bmatrix} v \\ 0 \end{bmatrix}, \qquad G = \begin{bmatrix} 0 \\ I_2 \end{bmatrix}, \nn
		\end{equation}
		where $m = 2$. For all $i \in \{1,\ldots,p\}$, the safe set is the complement of $p$ disk-shaped obstacles centered at $o_i \in \BBR^2$ with radius $r_i>0$,
		\begin{equation}
			h(x) = \min_{i \in \{1,\ldots,p\}} h_i(x), \qquad h_i(x) = \|q - o_i\|^2 - r_i^2, \label{eq:di_h}
		\end{equation}
		which is non-smooth at every state equidistant, in the $h_i$ sense, from two or more obstacles. Differentiating $h_i$ along \eqref{eq:di_dynamics},
		\begin{align}
			\Lf{f}{h_i}(x)        &= 2(q - o_i)^\rmT v, &
			\Lf{G}{h_i}(x)        &= 0, \nn \\
			L_f^2 h_i(x)          &= 2\|v\|^2, &
			\Lf{G}{L_f h_i}(x)    &= 2(q - o_i)^\rmT. \label{eq:di_lie}
		\end{align}
		Since $\Lf{G}{h_i} = 0$ while $\Lf{G}{L_f h_i} \neq 0$, each $h_i$ has relative degree two with respect to $u$. In contrast to a polytopic barrier, the second drift derivative $L_f^2 h_i = 2\|v\|^2$ is generically \emph{nonzero}, so the order-two seed of Theorem~\ref{thm:higher_rel_deg} is genuinely required: with $f^{[2]}(x) = 0$ for \eqref{eq:di_dynamics}, the seed $x_{\BBD_{(2)},f} = x + f(x)\,\epsilon$ produces
		\begin{equation}
			\tilde h_i\bigl(x_{\BBD_{(2)},f}\bigr) = h_i(x) + 2(q - o_i)^\rmT v\, \epsilon + \|v\|^2 \epsilon^2, \nn
		\end{equation}
		from which $L_f h_i$ and $L_f^2 h_i = 2\|v\|^2$ are recovered as $1!$ and $2!$ times the coefficients of $\epsilon$ and $\epsilon^2$, in a single forward evaluation, with the lexicographic routing of Theorem~\ref{thm:vertex_selection} applied at the outer $\min$. The control coupling $\Lf{G}{L_f h_i}$ in \eqref{eq:di_lie} is recovered, per Proposition~\ref{prop:coupling} with $r = 2$, from the bivariate seed \eqref{eq:bivariate_seed}: for each input column $g_j = \col_j(G)$, the seed is $z_{g_j} = (x + g_j \epsilon_2) + \tilde f(x + g_j \epsilon_2)\, \epsilon_1 \in \BBD_{(1,1)}^4$, and the evaluation of the unmodified program of $h_i$ returns the coefficient of $\epsilon_1 \epsilon_2$ equal to the $j$-th entry of $2(q - o_i)^\rmT$, in agreement with \eqref{eq:di_lie}, without symbolic differentiation.
		
		The exponential-CBF condition (See Remark~\ref{rem:ECBF_extension}) for each $\delta$-active constraint, with linear $\alpha_1(s) = c_1 s$ and $\alpha_2(s) = c_2 s$, is
		\begin{equation}
			2(q - o_i)^\rmT\, u + 2\|v\|^2 + (c_1 + c_2)\,2(q - o_i)^\rmT v + c_1 c_2\, h_i(x) \geq 0. \label{eq:di_ecbf}
		\end{equation}
		At a non-smooth state lying between two obstacles $i$ and $j$, both constraints are active and \eqref{eq:di_ecbf} imposes two inequalities on $u$, with coefficient vectors $2(q - o_i)^\rmT$ and $2(q - o_j)^\rmT$. Because these point along the two distinct obstacle bearings, they are linearly independent, so the simultaneous QP is feasible: the two-dimensional input has exactly enough authority to satisfy both relative-degree-two constraints at once. This is the feature a scalar input lacks, and it is the relative-degree analogue of the simultaneous enforcement certified by Theorem~\ref{thm:forward_invariance}.
		
		\textit{Numerical simulation.} We take $p = 2$ obstacles at $o_1 = (0, 1)$ and $o_2 = (0, -1)$ with $r_1 = r_2 = 0.95$, leaving a narrow gap at the origin. The robot starts at $q(0) = (-3, -1)$ with $v(0) = 0$ and is commanded to $q_{\rm goal} = (3, -0.2)$ by the nominal proportional-derivative law $u_{\rm n} = 1.5\,(q_{\rm goal} - q) - 2.5\, v$. The ECBF gains are $c_1 = c_2 = 4$, the margin is $\delta = 0.15$, and \eqref{eq:di_dynamics} is integrated by explicit Euler over $T = 14\,\mathrm{s}$ at step $\Delta t = 10^{-3}\,\mathrm{s}$. Figure~\ref{fig:di_traj} shows the robot threading the gap, and Figure~\ref{fig:di_h_active} reports $h(x(t))$ and $|\SA_\delta(x(t))|$. The composite barrier remains nonnegative, with minimum value $0.05$ inside the gap, and the active set is a singleton on approach but grows to two for the $1.46\,\mathrm{s}$ during which the robot passes between the obstacles. Over that interval the filter enforces both relative-degree-two constraints simultaneously through its two-dimensional input, combining the truncated-dual extraction of Theorem~\ref{thm:higher_rel_deg}, the vertex selection of Theorem~\ref{thm:vertex_selection}, and the simultaneous enforcement of Theorem~\ref{thm:forward_invariance} in a single non-smooth event. \exampletriangle
	\end{example}
	
	\begin{figure}[ht!]
		\centering
		\includegraphics[width=0.80\columnwidth]{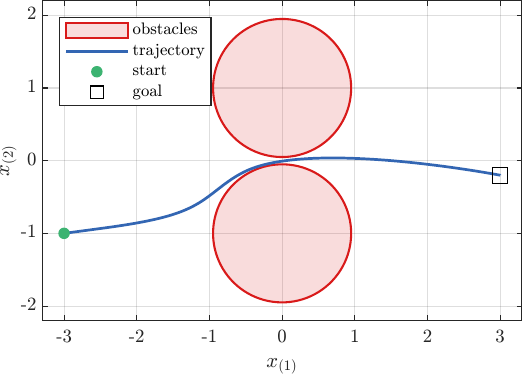}
		\caption{Closed-loop trajectory of Example~\ref{ex:doubleint}. The double integrator (blue) is commanded toward a goal $\square$. The nominal straight-line command drives the robot toward the obstacles (red). The relative-degree-two ECBF filter, acting on the two-dimensional acceleration, steers the robot safely through the gap.}
		\label{fig:di_traj}
	\end{figure}
	
	\begin{figure}[ht!]
		\centering
		\includegraphics[width=0.95\columnwidth]{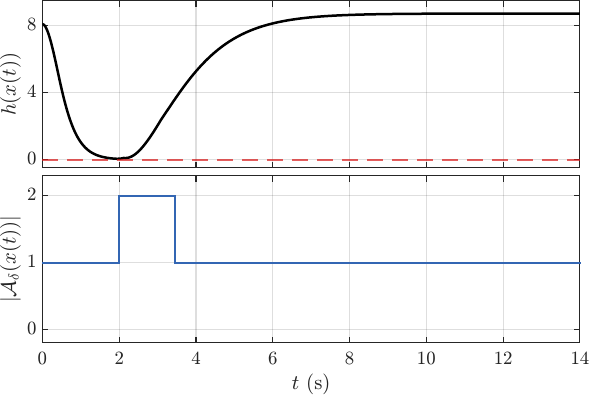}
		\caption{Composite barrier and active-set cardinality for Example~\ref{ex:doubleint}. Top: $h(x(t)) = \min_i h_i(x(t)) \geq 0$ throughout, dipping to $0.05$ inside the gap. Bottom: $|\SA_\delta(x(t))|$ equals two during the gap traversal, the non-smooth event at which two relative-degree-two constraints are enforced simultaneously.}
		\label{fig:di_h_active}
	\end{figure}
	
	\section{Discussion}
	\label{sec:discussion}
	
	\subsection{Comparison with Existing Approaches}
	
	Table~\ref{tab:comparison} compares the dual-algebraic approach with existing methods for non-smooth CBFs, with respect to the objectives \ref{P:invariance}--\ref{P:determinism} of this paper. Within this comparison, the dual-algebraic approach combines $O(C_h)$ worst-case time, deterministic computation, and exact subgradient extraction. Smooth surrogates attain $O(C_h)$ time and introduce approximation error; MIQP and hull-based methods are exact and deterministic, with computation times that depend on the state; sampling-based gradient estimates are not deterministic. Each row addresses the problem formulation for which it was developed; the comparison concerns their use for objectives \ref{P:invariance}--\ref{P:determinism}.
	
	\begin{table}[ht!]
		\centering
		\caption{Comparison of methods for non-smooth CBFs evaluated at a state $x$. Here $p$ denotes the number of smooth constraint functions $h_i$ composing $h = \min_i h_i$, $n$ is the state dimension, $C_h$ is the cost of evaluating $h$ in $\mathbb{R}$, and $|\mathcal{A}(x)|$ is the cardinality of the active set at $x$.}
		\label{tab:comparison}
		\footnotesize 
		\setlength{\tabcolsep}{3pt} 
		\renewcommand{\arraystretch}{1.2}
		\begin{tabular}{@{}p{3.2cm}lcc@{}} 
			\hline
			Method & Worst-case time & Det. & Exact \\
			\hline
			Smooth approx.\ \cite{lindemann2019control}           & $O(C_h)$                                    & Yes & No  \\
			Rand. smoothing \cite{duchi2012randomized}            & $O(C_h)$                                    & No  & No  \\
			Mixed-integer QP \cite{chen2018obstacle}    & $O(2^p)$                                    & Yes & Yes \\
			Clarke-hull enum.\ \cite{barber1996quickhull}      & $O(p + |\SA(x)|^{\lfloor n/2 \rfloor})$ & Yes & Yes \\
			\textbf{Dual-algebraic}                     & $\boldsymbol{O(C_h)}$                       & \textbf{Yes} & \textbf{Yes} \\
			\hline
		\end{tabular}
	\end{table}

	\subsection{Implementation Considerations}
	
	The dual-algebraic evaluation is implemented by storing each dual number as a pair $(a, b) \in \BBR^2$ and overloading the elementary arithmetic operations to follow Definition~\ref{def:dual_arith}. Languages with operator overloading (e.g., C++, Julia, Python) admit a transparent implementation; in C-only embedded environments, the dual-extended arithmetic can be inlined manually with no dynamic memory allocation, which is precisely the structural property required by safety standards such as DO-178C and ISO 26262. The composite-min evaluation \eqref{eq:dual_eval} reduces to $p - 1$ floating-point comparisons in addition to the $p$ dual evaluations of the individual $h_i$. When implemented in floating point, the comparison $\Re(z_1) \leq \Re(z_2)$ is the only branch in the procedure; near a real-part tie it may be sensitive to numerical noise, in which case the dual-part tie-break $\lexmin$ of Definition~\ref{def:lex_order} restores determinism, at the cost of a comparison no longer realized natively by the hardware floating-point comparator.

	\section{Conclusions and Future Work}
	\label{sec:conclusion}
	
	This paper presented a dual-algebraic framework for non-smooth control barrier functions defined as pointwise minima of smooth constraints, and showed that it meets the three objectives of Section~\ref{sec:problem}. Embedding the state and an arbitrary vector field into the dual-number ring $\BBD$ produces the value of the composite barrier and an exact Lie derivative from a single forward evaluation; the standard floating-point comparison, applied as $\minRe$, is a deterministic lexicographic selector that routes one vertex of the Clarke subdifferential to the QP at every state, including non-smooth points, which settles objectives~\ref{P:exactness} and \ref{P:determinism} (Theorem~\ref{thm:vertex_selection}, Proposition~\ref{prop:complexity}). Simultaneous enforcement of the $\delta$-active constraints then renders the safe set forward invariant, settling \ref{P:invariance} (Theorem~\ref{thm:forward_invariance}), whereas enforcing the selected vertex alone does not (Example~\ref{ex:counter}). The framework is closed under finite compositions of pointwise $\min$ and $\max$, both for gradient extraction (Theorem~\ref{thm:composition}) and for forward invariance under clause-wise enforcement (Theorem~\ref{thm:cnf_invariance}), and, through the truncated polynomial duals $\BBR[\epsilon]/(\epsilon^{r+1})$ and their bivariate extension, extends to systems of arbitrary relative degree, including extraction of the control coupling by elementary arithmetic (Theorem~\ref{thm:higher_rel_deg}, Proposition~\ref{prop:coupling}). The arithmetic overhead is a fixed constant factor of $C_h$, independent of $p$, $n$, $|\SA(x)|$, and the geometry of $\bd\SC$, which is the property required for hard real-time certification on embedded hardware. The three examples confirm the construction across relative degrees one and two and across polytopic, multi-agent, and nonlinear obstacle geometries.
	
	Several directions extend this work. The continuous-time guarantee of Theorem~\ref{thm:forward_invariance} carries to digital implementations once inter-sample drift is accounted for; a discrete-time dual-algebraic CBF inequality, in the spirit of \cite{breeden2021control}, would certify safety directly at the sample rate. Because the dual evaluation is exact and constant-cost regardless of how many constraints are simultaneously active, it pairs naturally with structure-exploiting QP solvers in cluttered environments, where the gradient extraction is no longer the bottleneck. Finally, the infinitesimal direction that carries the Lie derivative can equally carry a parameter sensitivity: seeding the dual part along a parameter-estimate direction would couple the framework with adaptive control barrier functions, allowing the safety certificate and the parameter update to share a single forward evaluation. Each of these directions preserves the defining feature of the present approach, that the safety-critical gradient information is produced exactly, deterministically, and in bounded time by elementary arithmetic alone.
	
	\bibliographystyle{IEEEtran}
	\bibliography{ref}
	

	\end{document}